\documentclass[a4paper,final]{siamart190516}

\usepackage{damacros}
\usepackage{booktabs}
\crefname{hypothesis}{Hypothesis}{Hypotheses}
\renewcommand{\phi}{\varphi}
\usepackage[caption=false]{subfig}
\usepackage{enumitem}
\usepackage{bm}

\headers{Projection methods for Neural Field equations}{D. Avitabile}

\title{Projection methods for Neural Field equations} 

\author{%
  Daniele Avitabile%
  \thanks{%
    Vrije Universiteit Amsterdam,
    Department of Mathematics,
    Faculteit der Exacte Wetenschappen,
    De Boelelaan 1081a,
    1081 HV Amsterdam, The Netherlands.
  \protect\\
    MathNeuro Team,
    Inria Centre at Universit\'e C\^ote d'Azur,
    2004 route des Lucioles-Boîte Postale 93 06902,
    Sophia Antipolis, Cedex, France.
  \protect\\
    (\email{d.avitabile@vu.nl}, \url{www.danieleavitabile.com}).
  }
}

\newcommand{\otoprule}{\midrule[\heavyrulewidth]}

\graphicspath{{./Figures/}}

\begin{document}

\maketitle

\begin{abstract}
  Neural field models are nonlinear integro-differential equations for the evolution
  of neuronal activity, and they are a prototypical large-scale, coarse-grained
  neuronal model in continuum cortices. Neural fields are often simulated
  heuristically and, in spite of their popularity in mathematical neuroscience,
  their numerical analysis is not yet fully established. We introduce generic
  projection methods for neural fields, and derive a-priori error bounds for these
  schemes. We extend an existing framework for
  stationary integral equations to the time-dependent case, which is relevant for
  neuroscience applications. We find that the convergence rate of a projection scheme for a
  neural field is determined to a great extent by the convergence rate of the
  projection operator. This abstract analysis, which unifies the treatment of
  collocation and Galerkin schemes, is carried out in operator form, without
  resorting to quadrature rules for the integral term, which are
  introduced only at a later stage, and whose choice is enslaved by the choice of the
  projector. Using an elementary timestepper as an example, we demonstrate that the
  error in a time stepper has two separate contributions: one from the projector, and
  one from the time discretisation. We give examples of concrete projection methods: two
  collocation schemes (piecewise-linear and spectral collocation) and two Galerkin
  schemes (finite elements and spectral Galerkin); for each of them we derive error
  bounds from the general theory, introduce several discrete variants, provide
  implementation details, and present reproducible convergence tests.
\end{abstract}

\section{Introduction}\label{sec:model}
Neural field models are integro-differential equations describing the
spatially-extended, coarse-grained activity of neurons in continuum cortices. The
simplest and most studied neural field model is written as
\begin{equation} \label{eq:NF}
  \begin{aligned}
  & \partial_t u(x,t) = -u(x,t) + \int_\Omega w(x,y) f(u(y,t))\, dy + \xi(x,t), \qquad
    && (x,t) \in \Omega \times J \\
  & u(x,t_0) = u_0,                         
    && x \in \Omega.
  \end{aligned}
\end{equation}
where $u(x,t)$ is the voltage of a neuronal population at time
$t$ and point $x$ in the cortex $\Omega$, which is a compact domain in
$\mathbb{R}^d$.
The function $w(x,y)$ is
the \emph{synaptic kernel}, modelling synaptic strengths from point $y$ to point
$x$ in the tissue, whereas $f$ is the \emph{firing rate} of the population; the
function $\xi$ models an external input. Finally, $J$ is a time interval
 containing $t_0$, typically $J = [t_0,t_0+T]$ or
$J=[t_0-T,t_0+T]$. Intuitively, the first term in the
right-hand side of the evolution equation in \cref{eq:NF} models local decay, while
the integral term collects inputs from the whole cortex. Typically the function $f$
is taken to be a steep, bounded sigmoidal function (approximating a Heaviside step
function), hence a neuronal patch in $y$ contributes to the activity in $x$
nonlinearly, and only if its voltage is sufficiently high.

Since their introduction by Wilson and Cowan~\cite{wilson1973mathematical}, and
Amari~\cite{amari1977dynamics}, neural fields have
been used to analyse and reproduce macroscopic cortical patters of activity,
including localised stationary bumps, travelling waves and spiral waves. As discussed
in several mathematical neuroscience reviews and textbooks 
\cite{Ermentrout.1998qno,Ermentrout.2010,Bresloff.2012,
Bressloff.2014k0p,coombes2014neural}, the model
\cref{eq:NF} can be extended in many ways, including multiple neuronal
populations, stochastic forcing, and distributed delays. 

Neural field equations support a wide variety of healthy and pathological neuronal
patterns; these relatively simple models are a popular choice for studying
macroscopic neural activity, and they display a dynamical repertoire observable in
more detailed, realistic cortical models
\cite{Bresloff.2012,Bressloff.2014k0p,coombes2014neural}. The study of neural fields
as dynamical systems is now well established, and the
literature contains mostly heuristic numerical simulations of these
models (in addition to the textbooks
\cite{Ermentrout.2010,Bressloff.2014k0p,coombes2014neural} and references therein,
see also
\cite{Folias.2005,hutt2008,Gils.2013,Rankin2013,kilpatrick2013wandering,laing2014numerical,Visser.2017,schmidt2020bumps})
and empirical convergence studies of selected numerical schemes
\cite{Rankin2013,martinNumericalSimulationNeural2018}.

By contrast, the numerical analysis of neural fields is much less 
developed. Intuitively, a numerical scheme can be derived by picking a spatial grid,
approximating the integral with a quadrature scheme, and using a time-stepper for the
corresponding set of ODEs. Convergence results are limited to these schemes (which we
will classify later as \textit{discrete collocation schemes}) in deterministic neural
fields \cite{Lima2015xt}, neural fields with anisotropic diffusion
\cite{Avitabile.2020}, stochastic neural fields \cite{lima2019numerical,kuehnLargeDeviationsNonlocal2014}, and neural
fields with distributed delays \cite{Faye.2010,polner2017}.

The present paper aims to provide an abstract framework for the development of
numerical schemes for neural field equations, thereby laying the foundations for a
systematic treatment of these models. Such theoretical developments are important:
(i) Numerical implementations of neural field models are becoming available in
dedicated software \cite{Leon2013,Nichols.2015,avitabile2020}, yet a generic numerical analytical
characterisation of the schemes employed in simulations is still missing. (ii) The Cauchy problem
\cref{eq:NF} is the archetype of models for cortical activity, being spatially extended,
nonlocal and nonlinear; hence, these equations are useful to prototype numerical
schemes for the neurosciences. (iii) It is hoped that a rigorous
numerical analysis of neural fields will help devising schemes for large scale
models, overcoming the expensive function evaluations which they currently require,
owing to the nonlocal term.

Our approach is to define an abstract function space setup for the numerical
treatment of \cref{eq:NF}, and study schemes in operator form. The strategy pursued in this article is to avoid the
discretisation of the neural field problem until the very last step. This is a useful
guiding principle in modern numerical analysis, as it allows to
understand properties of the scheme in great generality, prior to the spatial
discretisation step.

We characterise schemes for neural fields in terms of a projector $P_n$ from the
ambient Banach space $\XSet$ on which \cref{eq:NF} is posed, to a finite-dimensional
approximating subspace $\XSet_n$. We show that the choice of $\XSet$, $\XSet_n$, and
$P_n$ dictates the nature of the employed scheme and its convergence properties:
roughly speaking, projection schemes for \cref{eq:NF} converge at the same speed as a
projection $P_n v$ converges to $v$ in $\XSet$.

This approach has been used before for integral equations, but not for neural fields.
Steady states of \cref{eq:NF} satisfy a Hammerstein equation, for which projection
schemes have been studied rigorously in research articles by several authors
\cite{Atkinson:1973fe,Atkinson1987,kumar1987new,kumar1988discrete}
and detailed in excellent reviews and textbooks
\cite{Atkinson:1992du, atkinson2005theoretical,Chen:2015vv},
to which we refer for further references.
In the present paper we port this language to the time-dependent problem, and derive
a priori error bounds of projection schemes for \cref{eq:NF}, using primarily the
framework presented in books by Atkinson and coworkers
\cite{atkinson1997,atkinson2005theoretical} and by Chen, Micchelli, and
Xu~\cite{Chen:2015vv}.
The little cross-fertilisation between the numerical analysis and the mathematical
neuroscience community may be one reason why this step has not been taken to date,
and we hope that this paper will strengthen such link.

For neural fields, the abstract function space formulation allows to derive bounds
that are insightful with respect to the few existing convergence results available for
discrete schemes. For
instance, it emerges that the choice a quadrature scheme (which is made upfront in
discrete schemes for neural fields
\cite{Rankin2013,Lima2015xt,lima2019numerical,Avitabile.2020}) is in fact enslaved
by the choice of a projector $P_n$: the latter determines the accuracy that must be
matched by the former, when one commits to a discretisation of the problem. As we
will show here, a discrete scheme formulated without accounting for $P_n$ may waste resources, or pollute the
convergence of the projector scheme from which it is derived.

Also, the projector lets us classify methods into collocation and Galerkin schemes
(each with a Finite Element and Spectral variant), reconciling them with standard
PDE schemes, and suggesting schemes that are currently not used in the mathematical
neuroscience community.
Finally, the projector approach shows that, in a time stepping scheme, the error
splits naturally into a component due to the time discretisation, and a
separate one coming from $P_n$, as one would expect.

The paper is organised as follows: in \cref{sec:prelim} we cast \cref{eq:NF} as a
Cauchy problem on Banach spaces; projection methods of collocation and Galerkin type
are introduced in \cref{sec:projection_methods}, and their convergence properties are
studied in \cref{sec:convergence_projection_method}; concrete examples of how to use
the general convergence theory are presented in \cref{sec:examplesProjectionSchemes},
and convergence results on the Forward Euler time stepper are given in \cref{sec:timeStepper}.
\Cref{sec:numerics} presents numerical experiments, and \cref{sec:conclusions} concludes
the paper.

\subsection{Notation}
We denote by $C(\Omega,\RSet)$ or $C(\Omega)$ the space of
real-valued continuous function defined on $\Omega$, with the supremum norm $\|
\blank \|_\infty$. We use $L^2(\Omega,\RSet)$ or $L^2(\Omega)$ for the
Hilbert space of real-valued square-integrable functions defined on $\Omega$, with
norm $\| \blank \|_{L^2(\Omega)}$. Since we use both functional settings, we will
often consider $\XSet \in \{C(\Omega),L^2(\Omega)\}$, and write $\| \blank \|$, to
indicate the associated norm. 
We write $\| \blank \|_2$ for the $2$-norm on
$\RSet^n$. In addition, we denote by $B(\RSet)$ the space of bounded functions
defined on $\RSet$, with norm $\| \blank \|_\infty$.
We will also indicate by $BL(\XSet,\YSet)$ the space of bounded linear operators from
$\XSet$ to $\YSet$, where $\XSet$ and $\YSet$ are Banach spaces. The space
$BL(\XSet,\YSet)$ is endowed with the operator norm
\[
  \| A \|_{BL(\XSet,\YSet)} = \sup_{x \in \XSet_{\neq 0}} \frac{\| A
  x\|_\YSet}{\| x\|_\XSet}.
\]
We will also abbreviate $\| \blank \|_{BL(\XSet,\YSet)}$ by $\| \blank \|$, when the
context is unambiguous.
Since we work often with indices, it is useful to introduce the index sets 
$\NSet_n = \{1, \ldots,n\}$, $\ZSet_{n}= \{0,\ldots,n-1\}$, and $\ZSet_{\pm n}
=\{0,\pm 1,\ldots,\pm n\}$, for a fixed $n \in \NSet$.

\section{Neural field equations as Cauchy problems on Banach spaces} 
\label{sec:prelim} Before introducing the projection schemes, we wish to cast the neural
field equation
\cref{eq:NF} as a Cauchy problem on a Banach space $\XSet \in \{C(\Omega),
L^2(\Omega)\}$. These concrete choices for $\XSet$ are motivated by the functional
setup that are appropriate for collocation and Galerkin schemes, respectively.
We collect below our working hypotheses, which hold in standard neural field
models in literature~\cite{Ermentrout.2010,Bressloff.2014k0p,coombes2014neural}.
\begin{hypothesis}[Cortex]\label{hyp:domain}
  The domain $\Omega \subset \RSet^d$ is compact,
  with measure $| \Omega |$.
\end{hypothesis}
 
\begin{hypothesis}[Synaptic kernel]\label{hyp:kernel}
  If $\XSet = L^2(\Omega)$, the synaptic kernel $w$ is a function in $L^2(\Omega
  \times \Omega)$. If $\XSet = C(\Omega)$, we assume that the following holds: 
  \begin{align}
    & \lim_{h \to 0} \omega(h) = 0,  \qquad \omega(h) = \sup_{\substack{x,z \in \Omega\\ \| x- z \|_2 \leq h}} \int_\Omega
    |w(x,y) - w(z,y)|\, dy, \\
    & \sup_{x \in \Omega} \int_\Omega |w(x,y)| \, dy < \infty.
  \end{align}
\end{hypothesis}

\begin{hypothesis}[Firing rate]\label{hyp:firingRate}
  The firing rate function $f \colon \RSet \to \RSet$ is a bounded and everywhere
  differentiable Lipschitz function, hence $f,f' \in B(\RSet)$.
\end{hypothesis}

\begin{hypothesis}[External input]\label{hyp:externalInput} The mapping $t \to
  \xi(\blank,t)$ is continuous from $J$ to $\XSet \in \{C(\Omega),L^2(\Omega)\} $.This is written,
  with a slight abuse of notation\footnote{We are using the same symbol, $\xi$, to
  denote a function $\xi \colon \Omega \times J \to \XSet$, and the mapping $t \mapsto
\xi(\blank,t)$ on $J$ to $\XSet$. A similar notation will be also used for $u$.}, $\xi \in C(J,\XSet)$.
\end{hypothesis}

As we shall see in a moment, the hypotheses on the kernel guarantee the compactness
of a suitably defined integral operator $W$ with kernel $w$, which will be defined below \cite[Section
2.8.1]{atkinson2005theoretical}. In passing, we note that any $w \in C(\Omega \times
\Omega)$ satisfies \cref{hyp:kernel}. The firing rate $f$ is required to be
differentiable, thereby excluding the case of Heaviside firing rate, an hypothesis
that is employed in the analytic construction of neural field patterns
\cite[Chapters 1, 3]{amari1977dynamics,coombes2014neural}, but it is
dictated mostly by mathematical convenience, and is less relevant to numerical
simulations.

We define the following operators
\begin{align}
  F(u)(x) & := f(u(x)),               & & x \in \Omega, \label{eq:NemOp} \\ 
  (Wu)(x) & := \int_\Omega w(x,y) u(y)\, dy, & & x \in \Omega, \label{eq:WOp} \\
  K(u)(x) & := (WF(u))(x),         & & x \in \Omega, \label{eq:KOp}
\end{align}
and we rewrite \cref{eq:NF} formally, as
\[
  \begin{aligned}
    & u'(t) = N(t,u(t)):=-u(t) + K(u(t)) + \xi(t) \qquad & t \in J, \\
    & u(t_0) = u_0.
  \end{aligned}
\]
The system above is the sought Cauchy problem on the Banach space $\XSet$; this
section is devoted to make this step more precise, and to present bounds useful in
upcoming calculations.
%
%

%
As a preliminary step we collect a few results on the operators \crefrange{eq:NemOp}{eq:KOp}. We restate or combine results that are known in
literature, and we provide self-contained proofs (with reference to the relevant
papers) in the Appendix.

\begin{lemma}[Nemytskii operator]\label{lem:nemytskii}
  Let $\XSet \in \{C(\Omega),L^2(\Omega)\}$, and assume
  \cref{hyp:domain,hyp:firingRate}. Then 
  $F \colon \XSet \to \XSet$ is bounded and Lipschitz continuous, and
  \begin{equation} \label{eq:boundF}
   \|F(u)\| \leq \kappa_\Omega \|f \|_\infty \quad \text{for all $u \in \XSet$,}
  \end{equation}
  where 
  \begin{equation}\label{eq:kappaOmega}
   \kappa_\Omega = 
      \begin{cases}
	1              & \text{if $\XSet = C(\Omega)$,} \\
	|\Omega|^{1/2} & \text{if $\XSet = L^2(\Omega)$,} 
      \end{cases}
    \end{equation}
   and, in addition
   \begin{equation} \label{eq:LipF}
   \| F(u) - F(v) \| \leq \| f' \|_\infty \| u - v \| \qquad \text{for all $u,v \in \XSet$.}
 \end{equation}
\end{lemma}
\begin{proof}
  See \cref{app:proof:lem:nemytskii}.
\end{proof}

\begin{lemma}[Linear integral operator] \label{lem:W}
  Let $\XSet \in \{C(\Omega),L^2(\Omega)\}$, and assume
  \cref{hyp:domain,hyp:kernel}. Then $W \colon \XSet \to \XSet$ is a compact linear
  operator with
  \begin{equation}\label{eq:NormW}
    \| W \| \leq \kappa_w, \qquad \kappa_w = 
    \begin{cases}
      \displaystyle{\max_{x \in \Omega} \int_\Omega | w(x,y)| dy} & \text{if $\XSet = C(\Omega)$,} \\[1em]
      \| w \|_{L^2(\Omega \times \Omega)}          & \text{if $\XSet = L^2(\Omega)$.} 
    \end{cases}
    \end{equation}
    If $\XSet = C(\Omega)$, equality holds in \cref{eq:NormW}.
\end{lemma}
\begin{proof}
  See \cref{app:proof:lem:W}.
\end{proof}

\subsection{Cauchy problem on Banach spaces} We now return to the
Cauchy problem
\begin{equation}\label{eq:Cauchy}
  \begin{aligned}
    & u'(t) = N(t,u(t)) :=-u(t) + (WF)(u(t)) + \xi(t)  \quad & t \in J = [t_0-T,t_0+T], \\
    & u(t_0) = u_0.
  \end{aligned}
\end{equation}
As we have seen, \crefrange{hyp:domain}{hyp:externalInput} imply that $N \colon J
\times \XSet \to \XSet$, where $\XSet \in \{ C(\Omega), L^2(\Omega)\}$. We interpret
\cref{eq:Cauchy} as an ODE posed on $\XSet$. We say that $u$ is a solution to
\cref{eq:Cauchy} if: (i) The mapping $u \colon J \to \XSet$ is continuously
differentiable, that is, $u \in C^1(J,\XSet)$, and (ii) \Cref{eq:Cauchy} holds in $\XSet$
for all $t \in J$.

The existence and uniqueness of solutions to neural field equations, has been studied for
$\XSet =C(\Omega)$ by Potthast and beim Graben~\cite{Potthast:2010kb}, and for
$\XSet \in \{L^2(\Omega),H^m(\Omega)\}$ by Faugeras and coworkers
\cite{faugeras2008absolute,Veltz:2010}. There exist rigorous characterisations of neural
fields with delays \cite{Faye.2010,Visser.2017} and stochastic forcing
\cite{faugerasStochastic2015,maclaurin2020wandering}.
Here we present a self-contained proof in $\XSet \in
\{C(\Omega),L^2(\Omega)\}$, relying on the
Picard-Lindel\"of Theorem for the local existence and uniqueness of solutions to ODEs
posed on Banach spaces \cite[Theorem 5.2.4]{atkinson2005theoretical}. We begin with
a result on the operator $N$, which is instrumental for proving existence and
uniqueness of the solution.

\begin{lemma}\label{lem:continuityN}
  Let $\XSet \in \{ C(\Omega), L^2(\Omega)\}$, assume
  \crefrange{hyp:domain}{hyp:externalInput}, and let $Q_b$ be the set 
  \[
    Q_b = 
  \{ (t,u) \in \RSet \times \XSet \colon |t - t_0| \leq T, \; \| u -
    u_0 \| \leq b \}.
  \]
  For any $b>0$, the operator $N \colon Q_b \to \XSet$ is continuous, and Lipschitz
  continuous in its second argument, uniformly with respect to the first, with
  Lipschitz constant $1 + \kappa_w \| f' \|_\infty$, where $\kappa_w$ is given
  in~\cref{eq:NormW}.
\end{lemma}
\begin{proof}
  See \cref{app:proof:lem:continuityN}.
\end{proof}

The existence and uniqueness of \textit{classical solutions} to neural fields follows
directly from \cref{lem:continuityN} and the Picard-Lindel\"of Theorem.

\begin{theorem}\label{thm:existence}
  Let $\XSet \in \{ C(\Omega), L^2(\Omega)\}$, and assume
  \crefrange{hyp:domain}{hyp:externalInput}. For any $u_0 \in \XSet$
  there exists a unique solution $u \in C^1(J,\XSet)$ to~\cref{eq:Cauchy}.
\end{theorem}
\begin{proof}
  Fix $b>0$. By \cref{lem:continuityN} the operator $N \colon Q_b \to
  \XSet$ is continuous, and uniformly Lipschitz continuous with respect to its
  second argument, then by \cite[Theorem 5.2.4]{atkinson2005theoretical} there exists
  a unique solution $u \in
  C^1([t_0-\tau_0,t_0 + \tau_0],\XSet)$ to \cref{eq:Cauchy}, where
  \[
    \tau_0 = \min(T,b/M), \qquad M = \max_{(t,u) \in Q_b} \| N(t,u) \|.
  \]
  The argument above is valid for arbitrarily large $b$, hence we can take $b > MT$,
  which guarantees $\tau_0 = T$.
\end{proof}

\section{Projection methods}%
\label{sec:projection_methods}
We are now ready to discuss two families of schemes for the neural field equations,
the Collocation and the Galerkin method. They are projection methods, in the
definition classically used for integral
equations~\cite{atkinson1997,atkinson2005theoretical,Chen:2015vv}.
In an abstract projection method, one introduces a Banach space $\XSet$, and a
sequence of finite-dimensional approximating subspaces $\{ \XSet_n \colon n \in \NSet\}$ of $\XSet$, with $\overline{\bigcup_{n \in \NSet}
\XSet_n} = \XSet$. We denote the dimension of the approximating subspace $\XSet_n$ by
$s(n)$, where $s(n) \colon
\NSet \to \NSet$, with $s(n) \to \infty$ as $n\to \infty$.

Each projection method employs a family of \emph{projection
operators}
$\{ P_n \colon n \in \NSet\}$, with $P_n \in BL(\XSet,\XSet_n)$, defined by
$P_n v = v$ for all $v \in \XSet_n$, and chooses an approximation $u_n \in \XSet_n$
to $u \in \XSet$ for which the residual 
  \begin{equation}\label{eq:residualDef}
    r(u) = u' - N(\blank,u)
  \end{equation}
is small, 
in a sense that depends on the particular method under consideration. 
If $\{ \phi_i \colon i \in \NSet_{s(n)}\}$
is a basis for $\XSet_n$, then $u_n(t)$ can be written as
\begin{equation}\label{eq:basisDecomp}
  u_n(t) = \sum_{j \in \NSet_{s(n)}} a_j(t) \phi_j.
\end{equation}

We define an abstract projection method by setting
\[
  \begin{aligned}
    & P_n r(u_n(t)) = 0, & t \in J = [t_0-T,t_0+T], \\
    & P_n \big( u_n(t_0)-u(t_0) \big) = 0, &
  \end{aligned}
\]
that is,
\begin{equation}\label{eq:projCauchy}
  \begin{aligned}
  & u_n'(t) = P_n N(t,u_n(t)), \qquad t \in J =[t_0-T,t_0+T], \\
  & u_n(t_0) = P_n u_0.
  \end{aligned}
\end{equation}
System \cref{eq:projCauchy} is a Cauchy problem in the $s(n)$-dimensional Banach
space $\XSet_n$, and this evolution equation is useful to prove convergence results for abstract
schemes. As we shall see below, concrete choices of $\XSet$, $\XSet_n$, and the projector
$P_n$ lead to Cauchy problems on $\RSet^{s(n)}$ which are equivalent to
\cref{eq:projCauchy} and useful in numerical implementations. Such Cauchy problems in
$\RSet^{s(n)}$ may differ strongly between each other, as they depend
on the choice of $P_n$; one of the contributions of this paper is that it is possible
to analyse schemes in a unified manner using \cref{eq:projCauchy}, and to derive the convergence rate
of a concrete projection scheme from the abstract theory, which we will now present.

\subsection{Convergence of the abstract projection
method}\label{sec:convergence_projection_method}
Motivated by the previous discussion, we present convergence results for the abstract
projection scheme \cref{eq:projCauchy}. We begin by proving existence and
uniqueness of solutions to this problem.
\begin{theorem}\label{thm:existenceProjCauchy}
  Let $\XSet \in \{C(\Omega),L^2(\Omega)\}$, and assume
  \crefrange{hyp:domain}{hyp:externalInput}. Further fix $n \in \NSet$, and let $P_n$
  be a projection operator from $\XSet$ to $\XSet_n$, with $\dim \XSet_n < \infty$.
  For any $u_0 \in \XSet$ there exists a unique solution $u_n \in C^1(J,\XSet_n)$
to \cref{eq:projCauchy}.
\end{theorem}
\begin{proof}
  See \cref{app:proof:thm:existenceProjCauchy}.
\end{proof}

We are concerned with determining whether the solution $u_n(t)$ of the projection
scheme \cref{eq:projCauchy} converges to the solution $u(t)$ to the neural field
problem \cref{eq:Cauchy}. Since both $u$ and $u_n$ are in $C(J,\XSet)$, we are
interested in determining conditions under which 
\[
  \| u_n - u \|_{C(J,\XSet)} = \max_{t \in J} \| u_n(t) -u(t) \| \to 0 \qquad
  \text{as $n\to \infty$}.
\]
The following result addresses this question for generic projection schemes in neural
fields, and it relates the convergence of $u_n$ to $u$ in $C(J,\XSet)$ to the
pointwise convergence of the projection operator $P_n v \to v$ in $\XSet$.

\begin{theorem}[Convergence of the projection method]\label{thm:convergence}
  Let $\XSet \in \{C(\Omega),L^2(\Omega)\}$, and assume
  \crefrange{hyp:domain}{hyp:externalInput}. For all solutions $u_n$ and $u$ to
  \cref{eq:Cauchy,eq:projCauchy}, respectively, the following bounds hold:
  \begin{align}
    & \| u - P_n u \|_{C(J,\XSet)} \leq \alpha_n \label{eq:boundSolProj} \\
    \frac{1}{1+ \beta_n} & \|u - P_n u\|_{C(J,\XSet)} \leq  
    \|u - u_n\|_{C(J,\XSet)} \leq e^{\beta_n} \|u - P_n u\|_{C(J,\XSet)},
    \label{eq:sandwich} 
  \end{align}
  where
  \begin{align}
    \alpha_n  & = e^T \Big( \| u_0 - P_n u_0 \| +  
		T \| f \|_\infty \| W - P_n W\| + T \| \xi - P_n \xi \|_{C(J,\XSet)} \Big), \label{eq:alphaDef}\\
    \beta_n & = e^T T  \| P_n W \| \| f' \|_\infty. \label{eq:betaDef}
  \end{align}
  Further, assume there exist $p$, $n_0$ such that $\|P_n W\| \leq p$ for $n > n_0$, then:
  \begin{enumerate}
    \item 
      $u_n \to u$ as $n \to \infty$ in $C(J,\XSet)$ for all solutions $u$ to
      \cref{eq:projCauchy} if, and only if, $P_n v \to v$ as $n \to \infty$ in
      $\XSet$ for all $v$ in $\XSet$.
    \item 
      If convergence occurs, then $\| u- u_n\|_{C(J,\XSet)}$ and $\| u-
      P_nu\|_{C(J,\XSet)}$ converge to $0$ at exactly the same speed.
  \end{enumerate}
\end{theorem}
\begin{proof}

We begin by deriving the upper and lower bounds~\cref{eq:sandwich}. From
\cref{eq:Cauchy,eq:projCauchy}, and omitting dependence on $t$
\begin{equation*}
  \begin{aligned}
    (u_n - u)' & = P_n\big[ N(\blank,u_n) - N(\blank,u) \big] + P_n
    N(\blank,u) - N(\blank,u) & \\
		     & = P_n\big[ N(\blank,u_n) - N(\blank,u) \big] + ( P_nu -u)'
		     & \text{on $J$},
  \end{aligned}
\end{equation*}
hence, using the definition of $N$ and rearranging terms,
\begin{equation*}
    (u_n - u)' + (u_n - u) = ( P_nu -u)' + ( P_nu -u) +
		     P_n W\big[ F(u_n) - F(u) \big] \quad \text{on $J$}.
\end{equation*}
Integrating the previous identity against an exponential factor 
from $\min(t_0,t)$ to $\max(t_0,t)$, and applying initial conditions
\cref{eq:Cauchy,eq:projCauchy}, we arrive at
\begin{equation}\label{eq:conv_proof_id_0}
  u_n(t) - u(t) = P_nu(t) -u(t) + \int_{t_0}^t e^{s-t} P_n W \big[ F(u_n(s)) - F(u(s))
  \big]\, ds, \quad t \in J.
\end{equation}

We note that the identity above is valid for $t \geq t_0$ as well as
$t< t_0$, and can be used to find both the lower and the upper bound in
\cref{eq:sandwich}. For the upper bound, take norms in $\XSet$ and bound
$e^{s-t}$
from above by $e^T$, arriving at
\[
  \| u(t) - u_n(t) \| \leq \| u(t) - P_n u(t) \| +B_n  \bigg|  \int_{t_0}^t \| u(s)
  - u_n(s) \|\, ds \bigg| , \qquad t \in J,
\]
where $B_n  = e^T \| P_n W \| \| f' \|_\infty$. Using Gronwall's
inequality in the form given by Amann in
\cite[Chapter 2, Lemma 6.1]{amannOrdinaryDifferentialEquations2011} we obtain, for all $t \in J$,
\[
  \| u(t) - u_n(t) \| \leq \| u(t) - P_n u(t) \| + B_n \bigg| \int_{t_0}^t \| u(s) -
  P_nu(s) \| e^{B_n|t-s|} \, ds \bigg|.
\]
By \cref{thm:existence} $P_n u - u \in C(J,\XSet)$, hence
\[
  \begin{aligned}
  \| u - u_n \|_{C(J,\XSet)} 
  & \leq \| u - P_n u \|_{C(J,\XSet)} 
    \max_{t \in J} 
    \bigg(
      1 + B_n \bigg| \int_{t_0}^t e^{B_n|t-s|}\,ds \bigg|
  \bigg) \\
  & = \| u - P_n u\|_{C(J,\XSet)} 
  \max_{t \in [t_0-T,t_0+T]} e^{B_n|t-t_0|} 
  =: e^{\beta_n} \| u - P_n u\|_{C(J,\XSet)} 
  \end{aligned}
\]
which is the upper bound in \cref{eq:sandwich}. For the lower bound, return to
\cref{eq:conv_proof_id_0} and estimate
\[
  \| u(t) - P_n u(t) \| \leq \| u(t) - u_n(t) \| + B_n \bigg|  \int_{t_0}^t \| u(s)
  - u_n(s) \|\, ds \bigg|, \qquad t \in J.
\]
and, since $u - u_n \in C(J,\XSet)$,
\[
  \| u - P_n u \|_{C(J,\XSet)} \leq (1 + \beta_n) \| u - u_n \|_{C(J,\XSet)},
\]
which gives the lower bound in \cref{eq:sandwich}. We then proceed to prove
\cref{eq:boundSolProj}. From \cref{eq:Cauchy}
\[
  (u-P_n u)' + (u-P_n u) = (W-P_nW) F(u) + (\xi - P_n \xi)
\]
which gives
\[
  \begin{aligned}
  u(t) - P_n u(t) = e^{t_0-t} (u_0 - P_n u_0) 
  & + \int_{t_0}^t e^{s-t} (W-P_nW) F(u(s)) \, ds \\
  & + \int_{t_0}^t e^{s-t} \big( \xi(s) - P_n\xi(s) \big)\, ds, \qquad
  t \in J
  \end{aligned}
\] 
Taking norms in $\XSet$, bounding the exponentials by $e^T$ and recalling that
$\xi - P_n \xi \in C(J,\XSet)$, we obtain \cref{eq:boundSolProj},
\[
    \| u - P_n u \|_{C(J,\XSet)} 
      \leq 
      e^T \bigg(
      \| u_0 - P_n u_0 \|  
     + T \| W - P_n W\| \| f\|_\infty  + T \| \xi - P_n\xi \|_{C(J,\XSet)} \bigg)
     =: \alpha_n,
\] 

The sufficient condition of statement 1 can be proved without the condition
on $\| P_n W \|$. If $P_n v \to v$ for all $v \in \XSet$ then, by \cite[Lemma
12.1.3]{atkinson2005theoretical}, $\{P_n v\}$ is uniformly convergent for all $v$ in
the subset $\{u(t) \colon t \in J\} \subset \XSet$, which is compact because $J$ is
compact and $u$ continuous. Hence $\| u -
P_n u\|_{C(J,\XSet)} \to 0$ as $n\to \infty$. Further, by \cite[Theorem
12.1.4]{atkinson2005theoretical} the compactness of $W$ implies $\| W - P_n W \| \to
0$. This in turn implies that $\beta_n$ is convergent, and hence bounded by some
$\beta > 0$. We conclude that, if $P_n v \to v$ for all $v \in \XSet$ then, for any
solution $u$ to \cref{eq:Cauchy}
\begin{equation}\label{eq:boundSpeed}
    \| u_n - u\|_{C(J,\XSet)} \leq e^{\beta} \| P_n u - u\|_{C(J,\XSet)} 
			      \to 0 \quad \text{as $n \to \infty$}.
\end{equation}

Henceforth we use the hypothesis on $ \| P_n W \|$, which guarantees the existence of $n_0$
and $p$ such that $\| P_n W\| < p$ for all $n > n_0$. The latter implies that
$\beta_n$ is bounded by $\beta = p Te^T\| f' \|_\infty$ for all $n > n_0$. 

The necessary condition in statement 1 is proved by contrapositive, that is, we
prove that if there exists $z_0 \in \XSet$ for which $P_n z_0 - z_0$ diverges in
$\XSet$, then there exist solutions $z$, $z_n$ to 
\cref{eq:Cauchy,eq:projCauchy}, respectively, such that $z_n-z$ diverges in
$C(J,\XSet)$. If we take $z$ and $z_n$ to be the solutions to
\cref{eq:Cauchy,eq:projCauchy} with initial conditions $z_0$ and
$P_n z_0$, respectively, we obtain, using the lower bound in \cref{eq:sandwich}
\[
  \| z_0 - P_n z_0 \| \leq \| z - P_n z \|_{C(J,\XSet)} \leq (1 + \beta) \| z -
  z_n\|_{C(J,\XSet)} \quad \text{for all $n>n_0$},
\]
and since $z_0 - P_n z_0$ diverges in $\XSet$, then $z - z_n$ diverges in $C(J,\XSet)$.

The proof of statement 2 follows from \cref{eq:sandwich}, because by the hypothesis
on $\| P_nW \|$ 
\[
 \frac{1}{1+ \beta} \|u - P_n u\|_{C(J,\XSet)} \leq  
 \|u - u_n\|_{C(J,\XSet)} \leq e^\beta \|u - P_n u\|_{C(J,\XSet)}
 \quad \text{for all $n>n_0$}
\]
\end{proof}

  \cref{thm:convergence} holds for $J = [t_0,t_0+T]$, but in many cases one is
  interested in the forward problem, $J = [t_0,t_0+T]$. The theorem still holds in
  this case, with smaller constants $\alpha_n$ and $\beta_n$, as stated below:
  \begin{theorem}[Convergence of projection method for forward
    problem]\label{thm:convergenceForward}
    If the
    hypotheses of \cref{thm:convergence} hold on $J=[t_0,t_0+T]$, then
    \cref{thm:convergence} holds with 
  \begin{align}
    \alpha_n  & = \| u_0 - P_n u_0 \| +  
		T \| f \|_\infty \| W - P_n W\| + T \| \xi - P_n \xi \|_{C(J,\XSet)}, \label{eq:alphaDefPos}\\
    \beta_n & =  T  \| P_n W \| \| f' \|_\infty. \label{eq:betaDefPos}
  \end{align}
  \end{theorem}
  \begin{proof}
    The proof is almost identical to the one of \cref{thm:convergence}, in that
    \cref{eq:conv_proof_id_0} holds now for $J = [t_0,t_0+T]$. The constants
    $\alpha_n$ and $\beta_n$ differ from the ones in \cref{thm:convergence}: they do not
    display the factor $e^T$ because one can now bound the
    exponentials $e^{s-t}$ and $e^{t_0-t}$ in the proof of \cref{thm:convergence} by $1$, as opposed to $e^T$. 
  \end{proof}

\cref{thm:convergence} and its variant, \cref{thm:convergenceForward}, are the central results of the paper, and we make a few comments
on how they can be used when a concrete choice of $P_n$ is made, that is, when a
particular scheme is selected. There are two possible scenarios:
\begin{description}
  \item[Case 1]\label{item:scenario1} The projector is such that $P_n v \to v$ for all $v \in \XSet$. This covers
    several, but not all cases; in this circumstance convergence is ensured for all
    solutions $u$, at precisely the same speed as $\| u - P_n u\|_{C(J,\XSet)}$. As we shall
    see below, an estimate of $\| u - P_n u\|_{C(J,\XSet)}$ can often be obtained
    by studying the convergence of $\{P_n v\}$ in $\XSet$: it suffices to study the
    ``spatial" convergence rate of the projector operator to assess the
    ``spatiotemporal" convergence rate.
  \item[Case 2]\label{item:scenario2} The projector is such that $P_n v \to v$ fails for some $v \in \XSet$. In
    this case, the method does not converge for all solutions $u$. Convergence to
    \textit{certain} $u$ may still be possible though: convergence is
    guaranteed for problems in which $\|W - P_n W\|$, $\| u_0 - P_n u_0 \|$, and $\| \xi
    - P_n \xi\|_{C(J,\XSet)}$ tend to $0$ as $n \to \infty$.
    These conditions ensure that $\| P_n W \|$ and $\beta_n$ are bounded,
    and that $\alpha_n \to 0$ as $n \to \infty$, hence combining
    \cref{eq:boundSolProj,eq:sandwich} we have
    \[
      \| u - u_n\|_{C(J,\XSet)} \leq \alpha_n e^{\beta} \to 0 \qquad \text{as $n \to
      \infty$}.
    \]
    The asymptotic convergence rate of the scheme is the one of $\{\alpha_n\}$.
\end{description}

In passing, we note that an analysis of abstract discrete projection schemes 
seems possible: starting from \cref{eq:projCauchy} one can introduce a quadrature
scheme $Q_q$ with $q \in \NSet$ nodes, use it to define a new nonlinear problem
$u_{n,q}' = P_n N_q(\blank,u_{n,q})$ on $J$ with initial condition $u(t_0) = P_n
u_0$, and investigate whether the error bound $\| u - u_{n,q}\|$ splits in a
component proportional to the projection error, $\|u - P_n u \|$, and one
proportional to the $q$-dependent quadrature error. A useful framework for these
results is the theory of collectively compact operator approximations
\cite{Anselone.1971,Atkinson:1973fe}, albeit this avenue of research is not pursued
in the present paper.

We conclude this section by presenting bounds on the first and second derivative of
$u_n$, which are useful in upcoming calculations.
\begin{lemma}\label{lem:boundUnPrime}
  Let $\XSet \in \{C(\Omega),L^2(\Omega)\}$, and assume
  \crefrange{hyp:domain}{hyp:externalInput} hold for $J =[t_0,t_0+T]$ or
  $J=[t_0-T,t_0+T]$. For all solutions $u_n$ to
  \cref{eq:projCauchy} it holds
  \begin{equation}\label{eq:boundUnPrime}
    \| u'_n \|_{C(J,\XSet)} \leq \| u_n \|_{C(J,\XSet)} + \| P_n \xi \|_{C(J,\XSet)}
    + \gamma_n, \qquad \gamma_n = \kappa_\Omega \| P_n W\| \| f \|_\infty.
  \end{equation}
  where $\kappa_\Omega$ is defined in \cref{eq:kappaOmega}. If, in addition, $\xi \in
  C^1(J,\XSet)$, then $u_n \in C^2(J,\XSet)$ and
  \begin{equation}\label{eq:boundUnDoublePrime}
    \| u''_n \|_{C(J,\XSet)} \leq L_n \| u'_n \|_{C(J,\XSet)} + \| P_n \xi'
    \|_{C(J,\XSet)}, \qquad L_n = 1 + \beta_n/T,
  \end{equation}
  with $\beta_n$ given by \cref{eq:betaDef}. Further, if $P_n v \to v$ for all $v
  \in \XSet$, there exist positive constants $\kappa_1, \kappa_2$, independent of $n$,
  such that
  \begin{equation}\label{eq:boundUniformUnDiff}
    \| u'_n \|_{C(J,\XSet)} \leq \kappa_1, \qquad
    \| u''_n \|_{C(J,\XSet)} \leq \kappa_2.
  \end{equation}
  \begin{proof}
    See~\cref{app:proof:lem:boundUnPrime}.
  \end{proof}
  
\end{lemma}

\section{Examples of concrete projection
methods}\label{sec:examplesProjectionSchemes}

We now give examples of several projection methods, and corresponding estimates on
the convergence speed, showcasing the applicability of~\cref{thm:convergence}.

\subsection{Collocation method}\label{sec:collocationIntro}
For this scheme we set $\big( \XSet, \| \blank \| \big) = \big( C(\Omega), \| \blank \|_\infty \big)$ and
\[
  \phi_j = \ell_j, \qquad j \in \ZSet_{n+1}, \qquad n \in \NSet,
\]
where $\ell_j$ is the $j$th Lagrange interpolation polynomial with
nodes $\{x_j \colon j\in \NSet_{n+1}\}$, hence
\[
  \ell_i(x_j) = \delta_{ij}, \qquad i,j \in \ZSet_{n+1}.
\]
We introduce the spaces $\XSet_n = \spn\{\ell_0, \ldots, \ell_{n}\}$ with
dimensions $s(n) = n+1$, $n\in \NSet$ and the following family of operators
\begin{equation}\label{eq:collocationProj}
  P_nv = I_n v := \sum_{j \in \ZSet_{s(n)}} v(x_j) \ell_j,
  \qquad n \in \NSet.
\end{equation}
The operators above are a family of \emph{interpolatory projections} $\{ P_n \}$ from $\XSet$ to
$\XSet_n$ (from $C(D)$ to $\XSet_n$), for which we recall, without proof, the
following results~\cite[Section 12.1]{atkinson2005theoretical}:
\begin{proposition}\label{prop:interpProj} 
  Let $n \in \NSet$, and let $P_n$ be defined by \cref{eq:collocationProj}. Then $P_n
  \in BL(\XSet,\XSet_n)$ with
  \[
    \| P_n \|_{BL(\XSet,\XSet_n)} = \max_{x \in \Omega} \sum_{j \in \ZSet_{s(n)}}
    |\ell_j(x)| .
  \]
  Furthermore, for all $v \in \XSet$ we have
\[
  (P_n v)(x_i) = v(x_i), \qquad \text{for all $i \in \ZSet_{s(n)}$}.
\]
In addition, if $v \in \XSet$ then $P_n v = 0$ if, and only if, $(P_n v)(x_i) =
0$ for all $i \in \ZSet_{s(n)}$.
\end{proposition}
\Cref{prop:interpProj} shows that, in a collocation method, the abstract scheme
\[
  \begin{aligned}
    & P_n r(u_n(t)) = 0, & t \in J, \\
    & P_n \big( u_n(t_0)-u(t_0) \big) = 0, &
  \end{aligned}
\]
is equivalent to
\[
  \begin{aligned}
    & P_n r(u_n(t))(x_i) = 0, && i \in \ZSet_{s(n)}, \quad t \in J,  \\
    & P_n \big( u_n(t_0)-u(t_0) \big)(x_i) = 0, && i \in \ZSet_{s(n)}.
  \end{aligned}
\]
The two formulations above give rise to two equivalent $s(n)$-dimensional
evolution equations. The former leads to \cref{eq:projCauchy}, a Cauchy problem in
$\XSet_n$ which we used in \cref{sec:convergence_projection_method} to prove
convergence results. Using \cref{eq:basisDecomp} the latter system
gives\footnote{System~\cref{eq:CollocationODE} is the following set of approximating
  ODEs, in disguise
\[
    u'(x_i,t) \approx -u(x_i,t) + 
    \int_\Omega w(x_i,y) f\Bigg( \sum_{j \in s(n)} u(x_j,t) \ell_j(y)\Bigg)\, dy \qquad 
    u(x_i,0) = u_0(x_i) \qquad i \in \NSet_{s(n)},
\]
The latter formulation is possibly more directly relatable to \cref{eq:NF}, at a
first read.}
\begin{equation}\label{eq:CollocationODE}
  \begin{aligned}
    a'_i(t) &= -a_i(t) + N\bigg(t, \sum_{j \in \NSet_{s(n)}} a_j(t) \phi_j \bigg)(x_i),
		&& (i,t) \in \NSet_{s(n)} \times J, \\
    a_i(t_0) &= u_0(x_i),
		&& i \in \NSet_{s(n)},
  \end{aligned}
\end{equation}
that is, a Cauchy problem in $\RSet^{s(n)}$, which is useful for implementing the
scheme.


Different choices of the Lagrange interpolant and interpolation nodes give rise to
schemes with different properties. We discuss here
two families of schemes: (i) one where $\Omega$ is decomposed into elements $\Omega_i$,
and a local Lagrange interpolant is used (Finite-Element Collocation scheme); (ii)
one where interpolants are defined globally on $\Omega$ (Spectral Collocation
scheme). This treatment combines \cite{atkinson1997,atkinson2005theoretical} to
\cref{thm:convergence}.

\subsubsection{An example of Finite-Elements Collocation Method}\label{sec:FECollExample}
As a first example, we consider a piecewise-polynomial method (or finite-element
method). We decompose the domain $\Omega$ into elements $\Omega = \cup_{i \in
\NSet_n} \Omega_i$, and approximate $u \in C(\Omega)$ with piecewise polynomials with
local support. The functional setup for this scheme is  $\big( \XSet, \| \blank
\|\big) = \big( C(\Omega), \| \blank \|_\infty \big)$. We illustrate this method on a 1D domain $\Omega =
[-1,1] \subset \RSet$ on which we define a grid of $n+1$ points with mesh size $h_x(n) =
2/n$, as follows
\begin{equation}\label{eq:FECollGrid}
  x_i = i h_x, \qquad i \in \ZSet_{n+1}, \qquad \Omega_i = [x_{i-1},x_{i}], \qquad i \in \NSet_{n}.
\end{equation}
We approximate $u \in C(\Omega)$, with the classical shifted tent (piecewise linear)
functions, 
\begin{equation}\label{eq:FECollBasis}
  \ell_i (x) = 
  \begin{cases}
    \displaystyle{\frac{x - x_{i-1}}{x_i-x_{i-1}}} & \text{if $x \in [x_{i-1},x_i]$,} \\[1em]
    \displaystyle{\frac{x_{i+1} - x}{x_{i+1}-x_i}} & \text{if $x \in [x_{i},x_{i+1}]$,} \\[1em]
    0 & \text{otherwise.} 
  \end{cases}
  \end{equation}
  with adjustments for $\ell_0$ and $\ell_n$, which are supported on $[x_0,x_1]$ and
  $[x_{n-1},x_n]$, respectively. The functions $\{\ell_i\}$ form a Lagrange basis in
  that $\ell_i(x_j) = \delta_{ij}$. We take $\XSet_n =
\spn\{\ell_0, \ldots, \ell_n\}$, the space of all continuous piecewise-linear
functions on $\Omega$ with breakpoints $\{x_i \colon i \in \ZSet_{n+1}\}$, which has
dimension $s(n) = n+1$. We define the associated projector as
\begin{equation}\label{eq:FECollProj}
  P_n \colon \XSet \to \XSet_n, \qquad (P_nv)(x) =  (I_nv)(x) = \sum_{j \in \ZSet_{n+1}} v(x_j) \ell_j(x).
\end{equation}
The operator $P_n$ is an interpolatory projector at the nodes $\{x_i\}$, for which
the following bounds are known~\cite[Section 3.2.3]{atkinson2005theoretical}
\begin{equation}\label{eq:FECollProjBound}
  \| v - P_n v\|_\infty =
  \| v - I_n v\|_\infty \leq 
   \begin{cases}
     \omega(v,h_x),       & \text{if $v \in C(\Omega)$,} \\[0.5em]
     \displaystyle{\frac{h_x^2}{8} \| v'' \|_\infty}, & \text{if $v \in C^2(\Omega)$,} 
   \end{cases} 
\end{equation}
where $\omega$ is the modulus of continuity of $u$. 

The collocation finite-element method derived from $P_n$ is given by
\[
  \begin{aligned}
    a'_i(t) &= -a_i(t) + \sum_{j \in \NSet_{n}} \! \! \int_{\Omega_j}
    w(x_i,y) f\bigg(\sum_{k\in \ZSet_{n+1}} a_k(t) \ell_k(y)\bigg) \, dy + \xi(x_i,t),
		&& i \in \ZSet_{n+1}, \\
    a_i(t_0) &= u_0(x_i)
		&& i \in \ZSet_{n+1}.
  \end{aligned}
\]
where the integrals are taken over the elements $\Omega_i$. 

We can apply directly
\cref{thm:convergence}, and obtain the following convergence result.
\begin{corollary}[Convergence of the Finite-Element Collocation Scheme]\label{cor:convFEColl}
  Assume the hypotheses of \cref{thm:convergence} or
  \cref{thm:convergenceForward}, fix $\XSet=C^2(\Omega)$, and let 
  $P_n = I_n$ be given by \crefrange{eq:FECollGrid}{eq:FECollProj}. For any solution $u$
  to \cref{eq:Cauchy}, and $u_n$ to \cref{eq:projCauchy} it holds
  \begin{equation}\label{eq:FECollConv1}
  \| u - u_n \|_{C(J,\XSet)}  \to 0, \qquad \text{as $n \to \infty$}.
  \end{equation}
  If, in addition, $u \in C(J,C^2(\Omega))$ then there exists a constant $\kappa_u >0 $,
  dependent on $u$ but not on $n$, such that
  \begin{equation}\label{eq:FECollConv2}
    \| u - u_n \|_{C(J,\XSet)}  \leq \kappa_u h_x^2 \in O(n^{-2}) \quad \text{as $n \to
    \infty$}.
  \end{equation}
\end{corollary}
\begin{proof}
  By \cref{eq:FECollProjBound} we conclude that $P_n v \to v$ as $n \to \infty$ for
  all $v \in \XSet = C(\Omega)$. We are in Case 1 on page \pageref{item:scenario1}, and statement 1 of
  \cref{thm:convergence} (or \cref{thm:convergenceForward}) gives \cref{eq:FECollConv1}. 

  Let us now apply \cref{eq:sandwich} for a fixed $n \in \NSet$. Reasoning as in the
  proof of \cref{thm:convergence} (see discussion leading to \cref{eq:boundSpeed}),
  since $P_n v \to v$ for all $v \in C(\Omega)$, then $\beta_n$ is convergent, and
  hence bounded by a constant $\beta$. It holds 
  \[
    \| u - u_n \|_{C(J,C(\Omega))} \leq e^{\beta_n} \| u - P_n u\|_{C(J,C(\Omega))} \leq
    e^{\beta} \max_{t \in J} \| u (t) - P_n u(t) \|_\infty.
  \] 
  Under the hypothesis $u(t)\in C^2(\Omega)$ for all $t \in J$ we estimate, using
  \cref{eq:FECollProjBound},
  \[
      \| u - u_n \|_{C(J,C(\Omega))} 
		 \leq \frac{ e^\beta h_x^2}{8} 
		 \max_{t \in J} \bigg\| \frac{\partial^2
		 u(\blank,t)}{\partial x^2} \bigg\|_\infty
		 =: \kappa_u h_x^2.
  \]
\end{proof}

In \cref{sec:model}, we anticipated that the error bounds found in the projection
schemes are independent of quadrature schemes, and we can now see this in action. Concrete
implementations of this projection scheme require the choice of a quadrature rule to
approximate the integrals over the finite elements $\Omega_i$, in the variable $y$.
Following the classification in \cite{atkinson2005theoretical,Chen:2015vv}, a scheme
making such choice is a \textit{discrete projection scheme} (a
discrete collocation scheme in this case). 

The bound in \cref{cor:convFEColl}, however, shows that one can assess convergence of
the scheme \textit{before picking a quadrature rule}: the bound has a term in
$h_x$ which pertains only to the projector $P_n$. This implies that care must be
taken so that the quadrature scheme converges at the same rate as
the projector, as expressed by \cref{eq:FECollProjBound}: slower convergence rate in the
quadrature would degrade the rate \cref{eq:FECollConv2}, and faster quadrature rates
would be wasteful, as the $O(h_x^2)$ error of the projector would dominate the
quadrature error. We shall exemplify this phenomenon in \cref{sec:numerics}.

In addition, once the discrete collocation finite element method is written, the
corresponding initial-value problem must be solved introducing a time-stepping
scheme. An example of such analysis will also be given in operator form, without
invoking quadrature, in \cref{sec:timeStepper}. 

\subsubsection{An example of Spectral Collocation Method} \label{sec:SpecCollExample}
To exemplify the Spectral Collocation scheme we consider a neural field posed on
$\Omega = [-1,1] \subset \RSet$, and we use a Lagrange interpolating polynomial with Chebyshev node
distribution (also known as Chebyshev interpolant), which has spectral
convergence rates for smooth functions~\cite{Berrut2004,trefethen2019approximation}.
We consider Chebyshev points and the associated Lagrange basis
\begin{equation}\label{eq:ChebyInterpBasis}
  x_i = \cos \frac{i\pi}{n}, 
  \qquad
  \ell_i(x) = \prod_{\{j \in \ZSet_{n+1} \colon j \neq i\}} \frac{x-x_j}{x_i-x_j}, 
  \qquad 
  i \in \ZSet_{n+1},
\end{equation}
and construct the interpolatory projector
\begin{equation}\label{eq:ChebIntProj}
  P_n \colon \XSet \to \XSet_n, \qquad (P_nv)(x) =(I_nv)(x) = \sum_{j \in \ZSet_{n+1}} v(x_j) \ell_j(x).
\end{equation}
The spectral Chebyshev collocation method derived from $P_n$ is given by
\begin{equation}\label{eq:SpecCollCheb}
  \begin{aligned}
    a'_i(t) &= -a_i(t) + \int_{\Omega}
    w(x_i,y) f\bigg(\sum_{k\in \ZSet_{n+1}} a_k(t) \ell_k(y)\bigg) \, dy + \xi(x_i,t),
		&& i \in \ZSet_{n+1}, \\
    a_i(t_0) &= u_0(x_i)
		&& i \in \ZSet_{n+1}.
  \end{aligned}
\end{equation}
where the integrals are taken over the full domain $\Omega$. In spite of the
similarity with the Finite-Element collocation scheme, the Spectral Collocation
scheme requires a separate treatment. \Cref{eq:FECollProj,eq:ChebyInterpBasis}
look similar, but their convergence properties differ, in that the
underlying Lagrange basis $\{\ell_i\}$ is different. While for the former
$P_nv \to v$ for all $v \in \XSet$, this property
does not hold for the latter. It is known that, for $P_n$ defined by
\crefrange{eq:ChebyInterpBasis}{eq:ChebIntProj}
\begin{equation}\label{eq:auxCheb}
  \| v - P_n v\|_\infty = (2 + 2/\pi \log n) \| v - p_n^*\|_\infty
\end{equation}
where $p^*_n$ is the best approximation polynomial of degree $n$ to $v$ on
$[-1,1]$ (\cite[Theorem 2.1]{Battles.2004}), implying\footnote{From $v - P_n v =
  v-p_n^* - P_n(v-p^*_n)$ we have $\| v- P_n v\|_\infty \leq (1+\| P_n\|) \| v -
  p^*_n\|_\infty$, which combined with \cref{eq:auxCheb} gives $\| P_n \| = 1 +2 /
\pi \log n$.}
\[
  \| P_n \| = 1 +2 / \pi \log n.
\]
The Principle of Uniform Boundedness guarantees the existence of $v \in \XSet$ for which $P_n
v$ does not converge to $v$ and hence, by \cref{thm:convergence} there are solutions
$u$ to the neural field problem for which $u_n$ does not converge to $u$ in
$C(J,\XSet)$. We are therefore in Case 2, on page \pageref{item:scenario2}. The
following result shows that convergence is however ensured for problems with
sufficiently regular synaptic kernel $w$, initial solution $u_0$, and forcing term
$\xi$.

\begin{corollary}[Convergence of spectral Chebyshev collocation scheme]\label{cor:ChebCollConv}
  Assume the
  hypotheses of \cref{thm:convergence} or \cref{thm:convergenceForward}, fix
    $\Omega = [-1,1]$, $\XSet = C([-1,1])$, and let $P_n$ be
  given by
  \crefrange{eq:ChebyInterpBasis}{eq:ChebIntProj}, then:
  \begin{enumerate}
    \item If $u_0(x)$, $w(x,y)$, and $\xi(x,t)$ are differentiable $m$ times in
	$x$
	and $u_0^{(m)}$ is $\alpha$-H\"older continuous, $\partial_x^m w$ 
      is $\alpha$-H\"older continuous with respect to $x$ uniformly in $y$, and
      $\partial_x^m \xi$ 
      is $\alpha$-H\"older continuous with respect to $x$ uniformly in $t$,
      respectively,
      then 
      \[
	\|u - u_n\|_{C(J,\XSet)} \in O\bigg(\frac{\log n}{n^{m+\alpha}}\bigg) \qquad
	\text{as $n \to \infty$.}
      \]

    \item If there is an $m\geq 1$ such that
	$w(x,y)$ has an $m$th derivative of bounded variation in $x$ on $[-1,1]$
	  for all $y \in [-1,1]$,
	$u_0$ has an $m$th derivative of bounded variation on $[-1,1]$, and
	$\xi(x,t)$ has an $m$th derivative of bounded variation in $x$ on $[-1,1]$ for all $t \in J$,
      then
      \[
	\|u - u_n\|_{C(J,\XSet)} \in O(n^{-m}) \qquad \text{as $n \to \infty$.}
      \]
  \end{enumerate}
\end{corollary}
  \begin{proof}
    The proof straightforwardly adapts arguments in \cite[Section 3.2]{atkinson1997}
    to the case of Chebyshev polynomials. In this
    proof, the symbol $\kappa$ denotes a constant independent of $n$ that may assume
    different values in different passages. We begin by proving part 1 of the
    corollary. We estimate 
    \[
      \begin{aligned}
	\| u_0 - P_n u_0 \|_\infty \leq (1 + \| P_n\|) \| u_0 - u^*_0\|_\infty 
				   & =  (2 + 2/\pi \log n) \| u_0 - u^*_0\|_\infty \\
				   & \leq (2 + 2/\pi \log n) \frac{\kappa}{n^{m+\alpha}}
      \end{aligned}
    \]
    where the last bound is a consequence of the Jackson's theorem and the H\"older
    condition on $u_0^{(m)}$. A
    similar strategy is used to bound
    \[
      \begin{aligned}
	\| W - P_n W \| & = \max_{x \in [-1,1]} \int_{-1}^1 
		    \bigg| w(x,y) - \sum_{i \in \ZSet_{n+1}} w(x_i,y)\ell_i(x) \bigg| \,dy \\
                        & := \max_{x \in [-1,1]} \int_{-1}^1 
		    | w(x,y) - w_n(x,y) | \,dy.
      \end{aligned}
    \]
    For fixed $y$, we apply Jackson's theorem to bound $w(\blank,y) - w_n(\blank,y)$,
    and we use the fact that the H\"older condition on $\partial_x^m w$ holds
    uniformly in $y$:
    \[
      | w(\blank,y) - w_n(\blank,y) | \leq \| w(\blank,y) - w_n(\blank,y) \|_\infty
      \leq (2+2/\pi \log n) \frac{\kappa}{n^{m+\alpha}}
    \]
    hence
    \[
      \| W - P_n W \| \leq (2+2/\pi \log n) \frac{\kappa}{n^{m+\alpha}}.
    \]
    A similar argument gives 
    \[
      \|\xi - P_n \xi\|_{C(J,X)} = \max_{t \in J} \| \xi(\blank,t) - P_n
      \xi(\blank,t)\|_\infty \leq (2+2/\pi \log n) \frac{\kappa}{n^{m+\alpha}}.
    \]
    We can now apply \cref{thm:convergence}: the bounds above imply $\alpha_n \in
    O(n^{-(m+\alpha)} \log n)$; further, since $\| W - P_n W\| \to 0$ as $n \to
    \infty$ the sequence $\{\beta_n\}$ is bounded. We deduce
      \[
	\|u - u_n\|_{C(J,\XSet)} \leq e^\beta \alpha_n \in O\bigg(\frac{\log n}{n^{m+\alpha}}\bigg) \qquad
	\text{as $n \to \infty$.}
      \]
  Part 2 of the statement is proved in a similar way to part 1, and we will only
  sketch it for the sake of brevity: since $u^{(m)}_0$ is of bounded variation, then
  $\| u^{(m)}_0 - P_n u^{(m)}_0\|_\infty \in O(n^{-m})$ as $n \to \infty$
  (see \cite[Theorem 2.1]{Battles.2004} and references therein). A similar statement
  holds for $\| W - P_n W \|$ and $\| \xi - P_n \xi \|_{C(J,\XSet)}$, and a further
  application of \cref{thm:convergence} or \cref{thm:convergenceForward} gives the
  assert.
  \end{proof}

\subsection{Galerkin method}
We now return to the abstract projection scheme \cref{eq:projCauchy}, and
discuss specialisations of the projector that leads to Galerkin schemes, rather than
collocation schemes.

For the Galerkin scheme we set $\big( \XSet, \| \blank\| \big) = \big( L^2(\Omega),
\| \blank \|_{L^2(\Omega)} \big)$,  a Hilbert space with inner product $\langle
\blank, \blank \rangle = 
\langle \blank, \blank \rangle_{L^2(\Omega)}$. The method uses
\emph{orthogonal projection operators}, defined by
\begin{equation}\label{eq:orthoProj}
  \langle P_n u, v \rangle = \langle u , v \rangle \qquad \text{ for all $u \in
  \XSet$ and $v \in \XSet_n$}
\end{equation}
We recall, without proof, a few properties of the orthogonal projectors,
see~\cite[Proposition 3.6.9]{atkinson2005theoretical} and~\cite[Section
2.2.1]{Chen:2015vv}.
\begin{proposition}\label{prop:orthogonalProj}
  Let $n \in \NSet$, and let $P_n$ be defined by \cref{eq:orthoProj}. Then $P_n \in
  BL(\XSet,\XSet_n)$, with $\| P_n
  \|_{BL(\XSet,\XSet_n)} = 1$. Furthermore, for all $v \in \XSet$ we have
\[
  \| v - P_n v \| = \min_{z \in \XSet_n} \|v - z\|.
\]
In addition, let $\{\phi_i \colon i \in \NSet_{s(n)} \}$ be a basis for $\XSet_n$. If $u
\in \XSet$, then $P_n u = 0$ if, and only if, $\langle u,\phi_i \rangle = 0$ for all
$i \in \NSet_{s(n)}$. 
\end{proposition}

\Cref{prop:orthogonalProj} shows that, in a Galerkin method, the abstract
projection scheme
\[
  \begin{aligned}
    & P_n r(u_n(t)) = 0, & t \in J, \\
    & P_n \big( u_n(t_0)-u(t_0) \big) = 0, &
  \end{aligned}
\]
is equivalent to
%
\begin{equation} \label{eq:GalerkinWithR}
  \begin{aligned}
   & \langle r_n(t), \phi_i \rangle = 0, && i \in \NSet_{s(n)}, \quad t \in J, \\
   & \langle u(t_0) - u_0, \phi_i \rangle =0, && i \in \NSet_{s(n)}.
  \end{aligned}
\end{equation}

As for the collocation method, we obtain two equivalent $s(n)$-dimensional
evolution equations. The former formulation is, once again, \cref{eq:projCauchy},
while the latter is a Cauchy problem in $\RSet^{s(n)}$, useful in numerical implementations,
\begin{equation}\label{eq:GalerkinODEAlpha}
  \begin{aligned}
  \sum_{j \in \NSet_{s(n)}} \langle \phi_i,\phi_j \rangle a'_j(t) & = 
  \bigg\langle N\bigg(t, \sum_{j \in \NSet_{s(n)}} a_j(t) \phi_j \bigg), \phi_i \bigg\rangle,
    && (i,t) \in \NSet_{s(n)} \times J, \\
    a_i(t_0) & = \langle u_0, \phi_i \rangle, 
    && i \in \NSet_{s(n)}.
  \end{aligned}
\end{equation}

Like Collocation methods, Galerkin methods are also split in two families: (i) Galerkin Finite Element
methods, in which $\Omega$ is decomposed in finite elements $\Omega_i$, and
locally-supported polynomials are employed; (ii) Spectral Galerkin methods, in
which global polynomials are used.

\subsubsection{An example of Finite Element Galerkin Method}
\label{sec:FEGalerkinExample}
We take $\XSet = L^2(-1,1)$, $\Omega_i$ as in \cref{eq:FECollGrid}, and $\XSet_n =
\spn\{\ell_0,\ldots,\ell_n\} $, where $\ell_i$ is the shifted tent function
\cref{eq:FECollBasis} with $\supp \ell_i = \Omega_i \cup \Omega_{i+1}$ for $i \in
\NSet_{n}$, and $\supp \ell_i = \Omega_i$ for $i = \{0,n+1\}$. It can be shown (see
\cite[Section 3.3.1]{atkinson1997})
\[
  \| P_n \| = 1, \qquad \| v - P_n v \|_{L^2(-1,1)} \leq \sqrt{2} \omega(v,h_x(n)) \qquad
  \text{for all $v \in L^2(-1,1)$,}
\]
where $\omega$ is the modulus of continuity of $v$. Hence $P_nv \to v$ for every $v
\in C([-1,1])$. Owing to the density of $C([-1,1])$ in $L^2(-1,1)$, this implies $P_n v \to v$ for
all $v \in \XSet = L^2(-1,1)$, and we are hence in Case 1. The scheme is written as
\[
  \begin{aligned}
  \sum_{j \in \ZSet_{n+1}} \langle \ell_i,\ell_j \rangle a'_{j}(t) 
   = 
   &- \sum_{j \in \ZSet_{n+1}} \langle \ell_i,\ell_j \rangle a_{j}(t)
   + \int_{\supp \ell_i} \ell_i(x)\xi(x,t) \, dx \\
   & + \int_{\supp \ell_i} \!\!\!\!\!
     \ell_i(x) \int_\Omega w(x,y) f\bigg( \sum_{j \in \ZSet_{n+1}} a_j(t)
   \ell_j(y)\bigg) \, dy \,dx  
   & i \in \ZSet_{n+1}, \\
   \sum_{j \in \ZSet_{n+1}} \langle \ell_i, \ell_j \rangle a_j(t_0) 
    =& \int_{\supp \ell_i} \ell_i(x) u_0(x) \, dx
   & i \in \ZSet_{n+1}.
  \end{aligned}
\]
We note that the basis $\{\ell_i\}$ is not orthogonal but the matrix with components
$\langle \ell_i,\ell_j \rangle$, is sparse and tridiagonal \cite[Equation
12.2.21]{atkinson2005theoretical}.

Since $P_n v \to v$ for all $v \in \XSet$, using \cref{thm:convergence} one can
prove the analogous to \cref{cor:convFEColl} for this scheme. We conclude that $\|u -
u_n\|_{C(J,\XSet)} \to 0$ as $n \to \infty$.
%
Note that the ambient space for this scheme is $\XSet = L^2(\Omega)$ hence the result
above means
\[
  \max_{t \in J} \| u(t) - u_n(t) \|_{L^2(-1,1)} \to 0.
\]
If $u(t) \in C(J,C^2(\Omega))$,
then uniform bounds for the solution can be derived as follows:
\[
  \begin{aligned}
    \| u(t) - P_nu(t) \|_\XSet 
    & =  \min_{z \in \XSet_n} \|u(t) - z\|_{L^2(\Omega)} \\
    & \leq \| u(t) - I_nu(t) \|_{L^2(\Omega)} \\
    & \leq |\Omega|^{1/2}  \| u(t) - I_nu(t) \|_\infty \leq \kappa_u h^2_x \in
    O(n^{-2}),
  \end{aligned}
\]
where we have used the fact that the orthogonal projector $P_n$ minimises the
distance from $P_nu(t)$ to $u(t)$, that it differs from the interpolatory projector
$I_n$ of \cref{sec:FECollExample}, and that the latter satisfies the bound
\cref{eq:FECollProjBound} for $u(t) \in C^2(\Omega)$. The above considerations are
summarised in the following result.

\begin{proposition}\label{cor:convFEGalerkin}
  \cref{cor:convFEColl} holds for $\big( \XSet, \| \blank \|\big) = \big(
  L^2(\Omega), \| \blank\|_{L^2(\Omega)} \big)$, provided $P_n$ is the
  orthogonal projector on $L^2(\Omega)$ to $\spn\{\ell_0,\ldots,\ell_n\}$.
\end{proposition}

\subsubsection{An example of Spectral Galerkin
Method}\label{sec:SpectralGelerkinExample}
For an example of this scheme, we consider a neural field problem posed on a ring
which is a common choice in literature
\cite{Ermentrout.1998qno,Ermentrout.2010,Bresloff.2012,
Bressloff.2014k0p,coombes2014neural}. We consider the problem on $\XSet =
L^2(0,2\pi)$, the space of square-integrable functions on $(0,2\pi)$. 
We shall also assume that
the kernel $w(x,y)$ is $2\pi$-periodic in both variables, the forcing $\xi(x,t)$ is
periodic in $x$, and the initial condition $u_0(x)$ is $2\pi$
periodic. Instead of providing error bounds in a form of a theorem for this scheme
(they are similar to the ones found above), we present the arguments to derive them
when $\XSet = L^2(0,2\pi)$. We will also discuss how to derive
stronger uniform bounds in the space $C_p(2\pi)$, the space of continuous
$2\pi$-periodic functions. 

In the spatially-periodic case, a basis for the approximating space is the set of
$2n+1$ periodic functions 
$\{1,\sin x, \cos x,\ldots, \sin nx, \cos nx\}$. The analysis and calculations are
convenient if one transplants the problem on the space of complex-valued functions
$L^2((0,2\pi),\CSet)$, spanned by the equivalent basis $\phi_j(x) =
e^{ijx}$, for $j \in \ZSet_{\pm n}$. We therefore have $\XSet_n = \spn\{\phi_{-n}, \ldots,
\phi_n\}$ of dimension $s(n)= 2n+1$, and we use the natural orthogonal projector
\begin{equation}\label{eq:projSpecGalerkin}
  P_n \colon \XSet \to \XSet_n, \qquad (P_n v)(x) = \frac{1}{2\pi} \sum_{j \in
  \ZSet_{\pm n}}
  \langle v, \psi_j \rangle \psi_j(x).
\end{equation}
The basis is orthonormal, hence the spectral Galerkin method reads
\begin{equation}
  \begin{aligned}
  a'_{i}(t) 
   = 
   &- a_{i}(t)
   + \int_{\Omega} \phi_i(x) \bigg[ 
     \xi(x,t) +
   \int_\Omega w(x,y) f\bigg( \sum_{j \in \ZSet_{n+1}} a_j(t)
 \phi_j(y)\bigg) \, dy \bigg]^* \,dx \\
 a_i(t_0) = & \int_{\Omega} \phi_i(x) u^*_0(x) \, dx
  \end{aligned}
\end{equation}
for $i \in \ZSet_{n+1}$, where the asterisk denotes complex conjugation. Standard
convergence results on Fourier series are available~\cite{Canuto.2006}, ensuring $P_n
v \to v$ for all $v \in \XSet$. Hence convergence follows from
\cref{thm:convergence}. 
In addition, estimates on $\|P_n v - v\|$ exist for $v \in  H^r(2\pi)$, the
closure of $C_p(2\pi)$ under the inner product norm $\| \blank\|_{H^r(2\pi)}$ given
below~\cite[Section 5.1.2]{Canuto.2006}
\[
  \| P_n v - v\|_{\XSet} = \| P_n v - v\|_{L^2(0,2\pi)} \leq \frac{\kappa}{n^r} \| v
  \|_{H^r(2\pi)}, \qquad \| v\|^2_{H^r(2\pi)} = \sum_{j=0}^k \| v^{(j)}\|_{L^2(0,2\pi)}
\]
This implies that for solutions $u \in C(J,H^r(2\pi))$, the scheme converges with an $O(n^{-r})$ error, because
\cref{thm:convergence} gives
\[
  \begin{aligned}
  \| u - u_n \|_{C(J,\XSet)} 
  & \leq e^\beta \max_{t \in J} \| u(t) - P_n u(t) \|_{L^2(0,2\pi)} \\
  & \leq \frac{\kappa e^\beta}{n^r} \max_{t \in J} \| u(t) \|_{H^r(2\pi)} = \| u
  \|_{C(J,H^r(2\pi))} \frac{\kappa e^\beta}{n^r} \in O(n^{-r}).
  \end{aligned}
\]

Finding uniform bounds for solutions $u \in C(J,C_p(2\pi))$ is also possible, albeit
this takes us from Case 1 to Case 2: when the
projector \cref{eq:projSpecGalerkin} is on $C_p(2\pi)$ to $\XSet_n$, it is no longer
true that $P_n v \to v$ for all $v \in C_p(2\pi)$, because $\| P_n
\|_{BL(C_p(2\pi),\XSet_n)} \in O(\log n)$ ~\cite[Section
3.7.1]{atkinson2005theoretical}, and we no longer have $\|K-P_nK \| \to 0$, in
general. Similarly to what we have seen in \cref{sec:SpecCollExample}, we can assume
further regularity on the kernel $w$, and obtain convergence results analogous to
\cref{cor:ChebCollConv} which we omit for the sake of brevity (see also \cite[Section
12.2.4]{atkinson2005theoretical}).

\section{Time integrators} \label{sec:timeStepper}
The discussion in the previous sections concerned the approximation of solutions to 
the infinite-dimensional initial-value problem
\begin{equation}\label{eq:cauchy_rep_2}
  u' = N(t,u(t)), \quad t \in J, \qquad u(t_0) = u_0,
\end{equation}
that is, an ODE on $\XSet$, by means of solutions to the $n$-dimensional problem
\begin{equation}\label{eq:cauchy_proj_rep_2}
  u_n' = P_nN(t,u_n(t)), \quad t \in J, \qquad u_n(t_0) = P_n u_0,
\end{equation}
an ODE on $\XSet_n$. As discussed in \cref{sec:projection_methods},
the evolution equation on $\XSet_n$ can be expressed as system of ODEs in
$\RSet^{s(n)}$ suitable for numerical implementation, even though the ODE on
$\XSet_n$ is more convenient for the analysis. The $s(n)$ coupled ODEs must then be
solved numerically, using a timestepper, which introduces errors. 

In this section we demonstrate how this further approximation can also be handled in
operator form. We do not present a general theory, but rather show with
a simple time stepper that the cumulative error of the scheme has two contributions:
one component ascribable to the projection (to approximate
\cref{eq:cauchy_rep_2} by \cref{eq:cauchy_proj_rep_2}), and one to the specific
timestepper employed to solve \cref{eq:cauchy_proj_rep_2}. The proof of
\cref{thm:ForwEulerConvergence} gives an indication that this is
a general principle, valid for other time stepping schemes.
To fix the ideas, the problem is posed on  the time interval $J =
[t_0,t_0+T]$, which is partitioned using evenly spaced points $t_k = t_0 + k h$, and a
sequence of
approximations $\{ U_n(t_k) \}_k$ to $\{ u_n(t_k) \}_k$ is generated, starting from
$U_n(t_0) = u_n(t_0)$.
We aim to derive convergence results that relate $U_n(t_k)$ to the original solution
$u(t_k)$, and we seek for bounds of the following type
 \[
   \max_{t_k \in [t_0,t_0+T]} \| u(t_k) - U_n(t_k) \| \leq E_\text{timestep} +
   E_\text{proj.}
 \] 
As we shall see, this is achieved combining the convergence results in
\cref{sec:convergence_projection_method} for the spatial error, with standard
ODE techniques for the temporal error.

\subsection{Forward-Euler Projection methods}
We demonstrate this procedure on the simplest type of timestepper, the Forward Euler
method\footnote{This scheme is presented only for illustrative purposes, and we do not
recommend using it in numerical simulations, for the well known limitations of the
forward Euler scheme for ODEs. As we shall see below, we have used a Runge Kutta
scheme for concrete calculations.},
coupled to a generic projection scheme \cref{eq:cauchy_proj_rep_2}. We write
abstractly the scheme as follows
\begin{equation}\label{eq:Euler}
  U_n(t_{k+1}) = U_n(t_k) + h_t P_n N(t_k, U_n(t_k)), \qquad k \geq 0, \qquad U_n(t_0)
  = P_n u_{0}.
\end{equation}
In passing, we note that the operator $P_n N$ in the vectorfield of
\cref{eq:cauchy_proj_rep_2}, is on $J \times \XSet_n$ to $\XSet_n$. Standard convergence results for the Euler
scheme are available for ODEs on $\RSet^{s(n)}$, and are applicable to the
equivalent set of ODEs derivable for \cref{eq:cauchy_proj_rep_2}. We therefore
derive convergence results on the application of the Euler scheme to the abstract problem on
$\XSet_n$, and we expect that they will mirror the ones for $\RSet^{s(n)}$. 

\begin{theorem}[Convergence of the Forward-Euler Projection
  Scheme]\label{thm:ForwEulerConvergence}
  Let $\XSet \in \{C(\Omega),L^2(\Omega)\}$ and $J = [t_0,t_0+T]$. Assume
  \crefrange{hyp:domain}{hyp:externalInput}. Further, assume $u_n \in C^2(J,\XSet)$. For all
  solutions $u_n$ and $U_n$ to \cref{eq:cauchy_rep_2,eq:Euler}, respectively, it holds
  \begin{equation}\label{eq:boundEulerUn}
    \max_{t_k \in [t_0,t_0+T]} \| u(t_k) - U_n(t_k) \| \leq 
    \frac{e^{TL_n} - 1}{L_n} \tau_n(h_t)+
    e^{\beta_n} \| u - P_n u\|_{C(J,\XSet)}
  \end{equation}
  where
  \[
   \beta_n  = T \| P_n W \| \| f' \|_\infty, \qquad L_n = 1 + \beta_n/T, \qquad
  \tau_n(h) = \| u''_n\|_{C(J,\XSet)} h_t /2.
  \]
  Further, if $P_n v \to v$ for all $v \in \XSet$, then there exist positive
  constants $\kappa_x, \kappa_t$, independent of $n$, such that
  \begin{equation}\label{eq:bounEulerHom}
    \max_{t_k \in [t_0,t_0+T]} \| u(t_k) - U_n(t_k) \| \leq 
    \kappa_t h_t+
    \kappa_x \| u - P_n u\|_{C(J,\XSet)}.
  \end{equation}
\end{theorem}
\begin{proof}
  Let $m$ be the number of Euler steps necessary to go from $t_0$ to $t_0+T$, that is, the
  integer for which $t_m - t_0 \leq T$ and $t_{m+1} - t_0 > T$. In the proof it will
  hold $k \in \ZSet_{m}$ or $k \in \NSet_{m}$, depending on the equation.
  From \cref{thm:convergenceForward} we have
  \begin{equation}\label{eq:aux_0}
    \| u(t_k) - U_n(t_k) \| \leq e^{\beta_n} \| u - P_n u\|_{C(J,\XSet)}+ \| u_n(t_k) -
    U_n(t_k) \|.
  \end{equation}
  In order to bound $\| e_{n,k}\| = \| u_n(t_k) - U_n(t_k) \|$ we define the
  ancillary sequence 
 \[
   V_n(t_{k+1}) = u_n(t_{k}) + h_t P_n N(t_{k},u_n(t_{k})),
 \] 
 and note
 \begin{equation}\label{eq:aux_1}
   \| e_{n,k} \| \leq \| u_n(t_k) - V_n(t_k) \| + \| V_n(t_k) - U_n(t_k) \|.
 \end{equation}
 The first term is bounded as follows
 \begin{equation}\label{eq:aux_2}
 \begin{aligned}
   \| u_n(t_k) - V_n(t_k) \| 
   & = \| u_n(t_k) - u_n(t_{k-1}) - h_t P_n N(t_{k-1},u_n(t_{k-1})) \|  \\
   & = \Big\| \int_{t_{k-1}}^{t_k} P_n N(s,u_n(s)) \, ds - h_t P_n N(t_{k-1},u_n(t_{k-1})) \Big\| \\
   & = \Big\| \int_{t_{k-1}}^{t_k} \big( u'_n(s)  -  u'_n(t_{k-1}) \big) \, ds \Big\| \\
   & \leq \| u''_n (t_{k-1}) \| \int_{t_{k-1}}^{t_k} (s- t_{k-1})\, ds  \\
   & \leq \| u''_n \|_{C(J,\XSet)} \frac{h_t^2}{2} = h_t\tau_n(h_t),
 \end{aligned}
 \end{equation}
 where we used $u_n \in C^2(J,C(\Omega))$ and the Mean Value Inequality for nonlinear
 operator in Banach spaces~\cite[Proposition 5.3.11]{atkinson2005theoretical}.
 The second term in \cref{eq:aux_1} is written as 
 \[
   \begin{aligned}
     \| V_n(t_k) - U_n(t_k) \| = & \|u_n(t_{k-1}) + h_t P_n N(t_{k-1},u_n(t_{k-1})) \\
			         & - U_n(t_{k-1}) - h_t P_n N(t_{k-1},U_n(t_{k-1}))
				 \|. \\
   \end{aligned}
 \] 
 Since $U_{n}(t_0)=P_n u_0$ is in $\XSet_n$, then $P_n U_n(t_k) = U_n(t_k)$ for all $k$, hence
 bounding the terms on the right-hand side
 \begin{equation}\label{eq:aux_3}
   \| V_n(t_k) - U_n(t_k) \| \leq (1 + h_t L_n)\|u_n(t_{k-1})-U_n(t_{k-1}) \| = 
   (1 + h_t L_n)\|e_{n,k-1}\| .
 \end{equation}
 Combining \crefrange{eq:aux_1}{eq:aux_3} we obtain 
 \[
   \begin{aligned}
     \| e_{n,k}\| 
     & \leq (1 + h_t L_n)\|e_{n,k-1}\| + h_t \tau_n(h_t) \\
     & \leq (1 + h_t L_n)^k \|e_{n,0}\| + h_t \tau_n(h_t)
     \sum_{j=0}^{k-1}(1+h_tL_n)^j \\
     & = \frac{(1+h_tL_n)^k - 1}{L_n} \tau_n(h_t),
   \end{aligned}
 \]
 where we used $e_{n,0}=0$. From $(1+h_tL_n)^k \leq e^{k h_t L_n} \leq e^{TL_n}$, we obtain 
 \begin{equation}\label{eq:aux_4}
   \| e_{n,k} \| = 
   \| u_n(t_k) - U_n(t_k) \| \leq
   \frac{e^{TL_n}- 1}{L_n} \tau_n(h_t),
 \end{equation}
 whose upper bound is independent of $k$. The bound \cref{eq:boundEulerUn} is
 obtained combining \cref{eq:aux_0} with \cref{eq:aux_4}, and taking the maximum over
 $[t_0,t_0+T]$.

 Finally, the condition $P_n v \to v$ for all $v \in \XSet$ implies the boundedness
 of $\{ \beta_n \}$, hence the existence of $\kappa_x$. 
 \cref{lem:boundUnPrime} implies the existence of a
 positive constant $\kappa_2$ such that $\| u''_n\|_{C(J,\XSet)} <
 \kappa_2$, which, together with the boundedness of $\{ L_n\}$, implies the existence
 of $\kappa_t$.
\end{proof}

\cref{thm:ForwEulerConvergence} gives convergence results relatable to the ones in
\cref{thm:convergence}. The bound \cref{eq:boundEulerUn} shows that the
combined error of a Forward Euler time stepper and a projection scheme has one
component proportional to the projection error, and one component proportional to
$\tau_n(h)$, the global truncation error of the Euler scheme. In passing, we note
that the projection scheme affects, in general, also the component proportional to
$\tau_n$, through a prefactor that depends on $L_n$. 

As for
\cref{thm:ForwEulerConvergence}, there are two scenarios: if $P_n v \to v$ for all $v
\in \XSet$, then \cref{eq:bounEulerHom}, ensures that that the scheme converges to
first order in time, and at the same rate of $\| P_n u - u \|_{C(J,\XSet)}$ in space. 

If, on the other hand, $P_n v \to v$ fails for some $v$ in $\XSet$, then
convergence can still occur to certain solutions $u$; in this case, a possible
strategy is to prove convergence using \cref{eq:boundEulerUn}; one can show that $\|
W - P_n W \|
\to 0$, which implies the boundedness of $\{\beta_n\}$ and $\{L_n\}$; 
in this case, a bound on $ \| u_n''\|_{C(J,\XSet)}$ must be sought using
\cref{eq:boundUnPrime,eq:boundUnDoublePrime} in \cref{lem:boundUnPrime}.

One of the consequences of \cref{thm:ForwEulerConvergence} is that it is immediate to
assess convergence of the Forward Euler Scheme combined with any of concrete the
projection operators discussed in \cref{sec:examplesProjectionSchemes}. For instance,
we had found that the Finite-Element Collocation Scheme given by
\crefrange{eq:FECollGrid}{eq:FECollProj} converges as $O(h_x^2)$ in space. The
following result shows that combining this scheme with a Forward Euler in time we
achieve $O(h_t)$ convergence in time, and $O(h_x^2)$ in space. Results of this
type are currently presented in
literature for \textit{discrete schemes}, where quadrature rules are
prescribed~\cite{Lima2015xt,lima2019numerical,Avitabile.2020}. We show here that they
are a consequence of the theory presented in the previous chapters.

\begin{corollary}[Convergence of the Forward-Euler Finite-Element Collocation scheme]
  Let $\XSet \in \{C(\Omega),L^2(\Omega)\}$ and $J = [t_0,t_0+T]$. Assume
  \crefrange{hyp:domain}{hyp:firingRate}, and $\xi \in C^1(J,\XSet)$. Let $u$ be a
  solution to \cref{eq:Cauchy}, and $U_n$ be a solution of the Forward Euler scheme
  \cref{eq:Euler}, with $P_n$ given by \crefrange{eq:FECollGrid}{eq:FECollProj}.
  There exist positive constants $\kappa_t$, $\kappa_x$, independent of $n$, such
  that
  \[
    \max_{t_k \in [t_0,t_0+T]} \| u(t_k) - U_n(t_k) \| \leq 
     \kappa_t h_t + \kappa_x h_x^2.
  \]
  \begin{proof}
    The results follows directly from \cref{eq:bounEulerHom,cor:convFEColl}. Note
    that the constant $\kappa_x$ in the present theorem statement differs from the
    one in \cref{eq:bounEulerHom}.
  \end{proof}
  
\end{corollary}

 \section{Numerical Results}\label{sec:numerics}
\begin{figure}
  \centering
  \includegraphics[width=\textwidth]{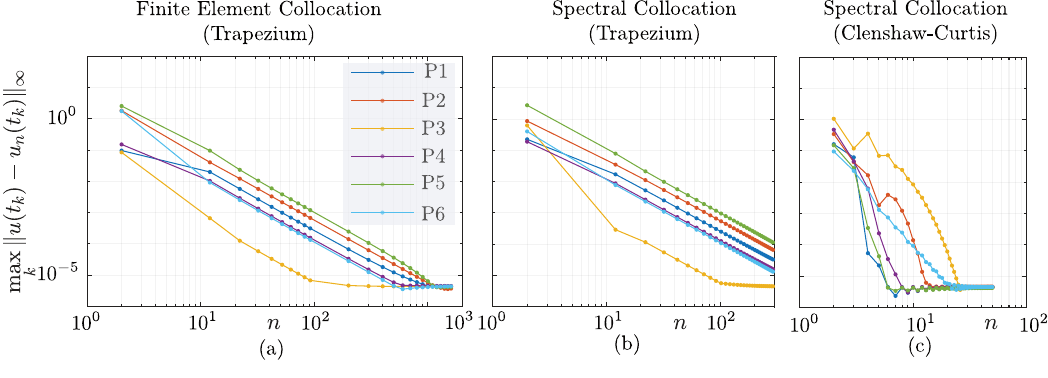}
  \caption{Convergence of (a) the Finite Element Collocation scheme of
    \cref{sec:FECollExample} with Composite Trapezium Quadrature, (b) the Spectral
    Collocation scheme of \cref{sec:SpecCollExample} with Composite Trapezium
    Quadrature, and (c) the same Spectral Collocation scheme with Clenshaw--Curtis
    quadrature, on test problems P1--P6. In (a) and (c) a quadrature scheme that matches the accuracy of the
projector has been selected, and gives convergence at the projector's rate. In (b) the
scheme displays $O(n^{-2})$ convergence, even though the projector has a
faster convergence rate, because the error is dominated by the $O(n^{-2})$ error of
the composite trapezium quadrature.}
  \label{fig:CollTests}
\end{figure}

\begin{figure}
  \centering
  \includegraphics[width=\textwidth]{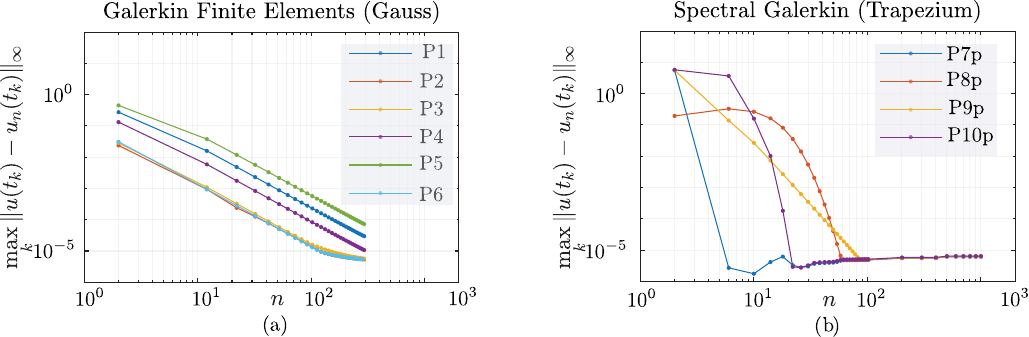}
  \caption{Convergence of (a) the Galerkin Finite Elements scheme of
    \cref{sec:FEGalerkinExample} with
    Gauss quadrature for problems P1--P6 and (b) the Spectral Galerkin scheme of
\cref{sec:SpectralGelerkinExample} on spatially-periodic problems P7p--P10p.}
  \label{fig:GalerkinTests}
\end{figure}

We tested the schemes described above using neural field equations with a solution in
closed form. We use
a common firing rate function function $f$ with explicit inverse, and a
kernel with product structure:
\[
  f(u) = \frac{1}{1+e^{-k(u-\theta)}}, \qquad
  f^{-1}(u) = \theta-\frac{1}{k}\log \frac{1-u}{u}, \qquad
  w(x,y) = \exp(-x^2+y^2) \zeta(y). 
\]
With these choices, $u_*(x,t) = f^{-1}( D \exp(-\gamma t - x^2))$, for $D \in (0,1)$,
solves the neural field problem on $\Omega = [-1,1]$ with external input given by
\[
  \xi(x,t) = \partial_t u_*(x,t) + u_*(x,t) - \zeta_0 f(u_*(x,t)),
  \qquad \zeta_0 = \int_{-1}^{1} \zeta(y) \, dy,
\]
and has therefore a closed-form expressions for suitable functions $\zeta(y)$. A
similar strategy was chosen for exact solutions to periodic problems. In particular,
we kept $f$ as before, and selected kernel and exact solutions as follows:
\[
  w_p(x,y) = \exp(-\cos^2 x+ \cos^2 y) \zeta_p(y), \qquad
  u_{*p}(x,y) = f^{-1}(D \exp(-\gamma t - \cos^2 x)),
\]
which give an exact solution to the neural field problem on $\Omega
= \RSet / 2\pi \ZSet$ for
\[
  \xi_p(x,t) = \partial_t u_{*p}(x,t) + u_{*p}(x,t) - \zeta_{0p} f(u_{*p}(x,t)),
  \qquad \zeta_{0p} = \int_{-\pi}^{\pi} \zeta_p(y) \, dy.
\]
By varying functions $\zeta$, $\zeta_p$ we obtained 10 test problems: 6 posed on
$\Omega = [-1,1]$ and labelled P1--P6, and 4 posed on
$\Omega = \RSet / 2\pi \ZSet$, labelled P7p--P10p. Parameters for the tests
are given in \cref{tab:exactProblems} in \cref{app:implementation}.
For the time discretisation we used an explicit Runge--Kutta (4,5) formula, the
Dormand--Prince pair implemented in Matlab's in-built \texttt{ode45} routine, with
default tolerance parameters. We tested several discrete schemes, by combining
Collocation or Galerkin schemes with quadrature rules. Detailed expressions for the
discrete schemes, and implementation details are given in \cref{app:implementation}.
Codes are hosted on a public repository, and all numerical
experiments in the paper can be modified and run with a single click \textit{without
  a Matlab license}, using the following coding capsule \cite{avitabile2021-codes}.

Before presenting the results, we recall that the numerical experiments
presented below use numerical schemes that slightly differ from the ones analysed in
\cref{sec:projection_methods} because: (i) a time stepper is employed for the time
discretisation and (ii) a quadrature rule is chosen to approximate integrals. The
discussions in \cref{sec:convergence_projection_method,sec:timeStepper} point to a
total error bounded by three contributions $E_\textrm{timestep}+ E_\textrm{projection} +
E_\textrm{quadrature}$. We do not yet have a convergence result for the
Runge--Kutta (4,5) pair implemented in Matlab, and for all the quadrature
rules presented below, but the convergence results in
\cref{sec:model,sec:projection_methods} predict convergence rates when the
timestepper error is negligible and the quadrature error
matches asymptotically the projection error, as we will now discuss.

\textit{Finite Element Collocation with Trapezium quadrature.}
We implemented the method described in \cref{sec:FECollExample}, with
$O(h_x^2)$ convergence, with the composite trapezium quadrature scheme, which
preserves this order of accuracy. This scheme is the most prominently used
in the mathematical neuroscience literature. \Cref{fig:CollTests}(a) shows the expected
$O(n^{-2})$ convergence of the $C(J,C^2([-1,1]))$ error for problems P1--P6. From
\cref{fig:CollTests}(a) we observe that the timestepper error dominates when the
$C(J,C^2([-1,1]))$ error is of the order of $10^{-5}$, and this will be true
henceforth for other experiments too.  \Cref{fig:CollTests}(a)
indicates that, when $E_\textrm{timestepper}$ is dominated by $E_\textrm{projection} +
E_\textrm{quadrature}$ and the latter errors are $O(n^{-2})$, then the total error is an
$O(n^{-2})$, as predicted by \cref{cor:convFEColl}. This is also confirmed by the
observation that, when the tolerance of the time stepper is tightened, the plateau in
\cref{fig:CollTests}(a) shifts from $10^{-5}$to a lower value (not shown, but verifiable via the
code capsule \cite{avitabile2021-codes}). 

\textit{Chebyshev Spectral Collocation with Trapezium and Clenshaw--Curtis quadrature.}
The scheme of \cref{sec:SpecCollExample} has been tested on problems P1-P6.
 For
these examples we expect a faster than quadratic convergence, for sufficiently smooth
kernels, provided the chosen quadrature scheme preserves this rate. In passing, we
note that the integrands $\zeta$ for P1--P6 are taken from Figure 2 in
\cite{Trefethen2008}, where the accuracy of Clenshaw--Curtis quadrature is
analysed for such functions. 
We first implemented a discrete scheme with a composite trapezium rule. From
\cref{fig:CollTests}(b) it is clear that the $O(n^{-2})$ rate of the quadrature
pollutes the overall convergence. 
We then switched to Clenshaw--Curtis quadrature, which uses Chebyshev points also as
quadrature nodes, has excellent convergence properties for this
setup~\cite{Trefethen2008}, and can be implemented with Fast Fourier Transforms
(FFTs). \cref{fig:CollTests}(c)
shows the superior convergence properties of this scheme,  as predicted by
\cref{cor:ChebCollConv}. At the time of writing we
are unaware of research papers where neural fields are simulated using the Chebyshev
Spectral Collocation with Clenshaw--Curtis quadrature, which is the most accurate and
efficient scheme between the ones presented here for non-periodic domains. 

\textit{Finite Element Galerkin with Trapezium and Gauss quadrature.} We derived from
the $O(h_x^2)$-convergent scheme of \cref{sec:FEGalerkinExample} two discrete
schemes, and tested them on P1-P6. In the first one we use composite Trapezium
quadrature to approximate the integral operator as
well as \emph{all inner products}, including the ones for the mass matrix $\langle
\ell_i,\ell_j \rangle$ which could be computed in closed form. In
\cref{app:implementation:FEGalerkin} we explain that the scheme so derived does not
require inner product evaluations, and in fact coincides with the
discrete Finite-Element Collocation scheme seen in \cref{fig:CollTests}(a).
We also derived a second discrete scheme, where the mass matrix is in closed form,
(and must be inverted at every function evaluation) and we use Gauss quadrature on reference
elements with $2$ nodes. This scheme's $O(n^{-2})$ convergence rate is seen in
\cref{fig:GalerkinTests}(a), even though it is less efficient than the one with Trapezium
quadrature which does not require inner products, as explained in
\cref{app:implementation:FEGalerkin}.  The findings in
\cref{fig:GalerkinTests}(a) are thus in line with \cref{cor:convFEGalerkin}.

\textit{Spectral Galerkin scheme with Trapezium quadrature.} 
Finally we derived a
discrete spectral Galerkin scheme from \cref{sec:SpectralGelerkinExample}, using
trapezium quadrature, which is well suited for periodic
integrands~\cite{trefethen2014exponentially}. We proceed with a pseudospectral
evaluation of the right-hand side, which can be performed with a forward and backward
FFT call. The fast convergence of the scheme is reported in
\cref{fig:GalerkinTests}(b) for periodic problems P7p--P10p (see discussion in
\cref{sec:SpectralGelerkinExample} for expected convergence rates). In passing we note that
further efficiency savings can be obtained if the neural field has a convolutional
structure, as it was shown in \cite{Rankin2013} for a collocation scheme with
pseudoscpetral evaluation of the right-hand side.

\section{Conclusions}\label{sec:conclusions} We have shown that projection methods in
use for Fredholm integral equations can be employed successfully in time-dependent
neural field equations. As in the stationary theory, convergence properties of the
projector determine the convergence rate of the scheme, and guide the choice of
quadrature rules in discrete methods. 

The theory presented here is applicable to generic domains in $\mathbb{R}^d$, and we
envisage that further extensions may lead to the adoption of projection schemes on
realistic cortices. In particular, it seems straightforward to adapt the methods
described here to the case of multiple neuronal populations. This requires the
definition of the problem on a different Banach space with respect to the ones
adopted here~\cite{Faugeras:2009gn}, and involves a bounded linear operator $L$ in place
of the operator $-\id_\XSet$ which gives the linear part of \cref{eq:NF}. The
adaptation seems to require the use of a uniformly continuous
semigroup $e^{Lt}$ in place of $e^{-t}$, used in this paper. In addition, with
suitable modifications, we envisage that projection methods can be used for neural
fields of new generation which have a different nonlocal evolution equation, but have
already been simulated with collocation or Galerkin
schemes~\cite{byrne2019next,schmidt2020bumps}.

It may also be possible to extend the projection method characterisation to neural fields with
delays for which discrete Galerkin methods exist~\cite{polner2017}.
When delays are present, the initial Cauchy problem \cref{eq:Cauchy} with $t_0=0$,
for instance, is replaced by a functional equation in
$C([-\tau,0],\XSet)$~\cite{visserSpikingNeuronsBrain2013,Gils.2013,polner2017}, with
$\tau$ being a maximal delay. A theory that blends projection methods and recent
progress on sun-star calculus
\cite{janssensClassAbstractDelay2019,janssensClassAbstractDelay2020} is unexplored,
nontrivial, and relevant for applications.

The adoption of projection methods on large scale problems continues to pose the
challenge of evaluating right-hand sides with large and dense matrices. One direction
that we are currently investigating is the adoption of multi-resolution bases
\cite{Chen:2015vv}, which lead to fast methods for stationary problems, and can
seemingly be ported to neural fields, thereby requiring only $O(n \log n)$ function
evaluations. Also, we have not investigated in this paper the stability of
timesteppers for neural fields, or a posteriori error bounds, which are important for
spatial and temporal adaptation of the schemes. We hope that this article will
stimulate the development of such techniques, and the rigorous study of numerical
approximations for spatially-extended neuroscience problems.

\section*{Acknowledgements} This article is dedicated to the memory of Prof.
Kenneth Andrew Cliffe. I am grateful to Jan Bouwe van den Berg, Lukas Bentkamp,
Stephen Coombes, Paul Houston, Gabriel Lord, Simona Perotto, and Sammy Petros, for
discussions that improved the presentation of the paper. I am particularly
grateful to an anonymous reviewer, whose comments led me to extend the results derived for the
forward problem $J=[t_0,t_0+T]$ to $J = [t_0-T,t_0+T]$, and who suggested to
restructure the paper so as to give more prominence to the abstract results.

\appendix

\section{Proof of \texorpdfstring{\cref{lem:nemytskii}}{Lemma Nemytskii operator}}\label{app:proof:lem:nemytskii}
\begin{proof}
  Let $\XSet = C(\Omega)$. If $u \in C(\Omega)$, then by \cref{hyp:firingRate} 
  $f(u) \in C(\Omega)$ and $|f(u)| \leq \|f\|_\infty$, hence $F \colon \XSet \to
  \XSet$, $F$ is bounded, and \cref{eq:boundF} holds. To prove that $F$ is Lipschitz,
  fix $x \in \Omega$; by \cref{hyp:firingRate} and the Mean Value Theorem there exists $z(x) \in
  (u(x),v(x))$ such that 
  \begin{equation}\label{eq:interm}
    | f(u(x)) - f(v(x)) | \leq |f'(z(x))| \; | u(x) - v(x) | \leq \| f' \|_\infty |
    u(x) - v(x) |,
  \end{equation}
  thus~\cref{eq:LipF} holds with $\| \blank \| = \| \blank \|_\infty$.

  Let us now turn to the case $\XSet = L^2(\Omega)$. To prove that $F$ maps
  $L^2(\Omega)$ to itself see for instance~\cite[Lemma 3.30]{Lord:2014ir}: for any $u \in
  L^2(\Omega)$, we have
  \[
    \Vert F(u) \Vert^2_{L^2(\Omega)} = \int_\Omega |f(u(x))|^2\, dx \leq |\Omega| \;
    \Vert f \Vert^2_\infty.
  \]
  thus $F \colon L^2(\Omega) \to L^2(\Omega)$, and \cref{eq:boundF} holds. 
  From \cref{eq:interm} we obtain
  \[
    \begin{aligned}
      \| F(u) - F(v) \|^2_{L^2(\Omega)} & = \int_\Omega | f(u(x)) - f(v(x)) |^2 \, dx  \\
                          & \leq \| f' \|^2_\infty \int_\Omega | u(x) - v(x) |^2 \, dx \\
			  & = \| f' \|^2_\infty \| u - v \|^2_{L^2(\Omega)}
    \end{aligned}
  \]
  which gives \cref{eq:LipF}.
\end{proof}

\section{Proof of \texorpdfstring{\cref{lem:W}}{Lemma linear integral operator}}\label{app:proof:lem:W}
\begin{proof}
  The operator $W$ is clearly linear. If $\XSet = C(\Omega)$, then \Cref{hyp:domain}
  implies compactness (hence boundedness) of $W \colon C(\Omega) \to C(\Omega)$ and
  $\| W \| = \kappa_w$ (see \cite[Section 2.8.1]{atkinson2005theoretical} and
  \cite[Section 1.2]{atkinson1997}).

  We then turn to the case $\XSet = L^2(\Omega)$. By \cref{hyp:kernel} $w \in
  L^2(\Omega \times \Omega)$, hence $W$ is a Hilbert-Schmidt operator, and this
  implies the compactness of $W$ (see, for instance \cite[Section 1.2]{atkinson1997}
  or \cite[online Chapter 8]{heil2018metrics}).
  Take $v \in L^2(\Omega)$ and set $z = Wv$. Since $w \in L^2(\Omega \times \Omega)$,
  then $w(x,\blank) \in L^2(\Omega)$ for almost all $x \in \Omega$. Thus, $w(x,\blank) v(\blank)$ is integrable
  (and $z$ well defined) for almost all $x \in \Omega$. Using
  standard definitions and the Cauchy-Schwarz inequality we have
\[
  \begin{aligned}
    \| Wu \|^2_{L^2(\Omega)} & = \int_\Omega \bigg| \int_\Omega w(x,y) u(y) \, dy \bigg|^2 \, dx  \\
	       & = \int_\Omega \langle w(x,\blank), u \rangle^2_{L^2(\Omega)} \, dx \\
	       &  \leq \int_\Omega \| w(x,\blank) \|^2_{L^2(\Omega)} \| u \|^2_{L^2(\Omega)} \, dx \\
	       & = \| u \|^2_{L^2(\Omega)} \int_\Omega \int_\Omega |w(x,y)|^2  \, dy \, dx 
	       = \| w \|^2_{L^2(\Omega \times \Omega)} \| u \|^2_{L^2(\Omega)},
  \end{aligned}
\]
which gives $\kappa_w$ when $\XSet = L^2(\Omega)$.
\end{proof}

\section{Proof of \texorpdfstring{\cref{lem:continuityN}}{Lemma Continuity of N}}\label{app:proof:lem:continuityN}
\begin{proof}
  Fix $b>0$, and consider a sequence $\{ (t_n,u_n)
  \}_{n \in \NSet} \in Q_b$ such that $(t_n,u_n) \to (t,u)$ as $n \to \infty$. We
  prove the continuity of $N$ by showing $N(t_n,u_n) \to N(t,u)$ as
  $n \to \infty$ in $\XSet$, that is, for any $\eps >0$, there exists an
  integer $m$ such that
  \[
    \| N(t_n,u_n) - N(t,u) \| \leq \eps \quad \text{for all $n \geq m$.} 
  \]
  Fix $\eps >0 $. The convergence of $(t_n,u_n)$ to $(t,u)$, the Lipschitz
  continuity (and hence continuity) of $F$ (see~\cref{lem:nemytskii}), and
  \cref{hyp:externalInput} imply the existence of
  integers $m_1$, $m_2$, $m_3$ such that
  \[
    \begin{aligned}
    & \| u_n - u \| \leq \eps / 3,
    &&\text{for all $n \geq m_1$,}\\
    & \| F(u_n) - F(u) \| \leq \eps / (3 \kappa_w), 
    &&\text{for all $n \geq m_2$,}\\
    & \| \xi(t_n) - \xi(t) \| \leq \eps / 3,
    &&\text{for all $n \geq m_3$,}
    \end{aligned}
  \]
  respectively, hence for all $n \geq m = \max(m_1,m_2,m_3)$.
   \[
       \| N(t_n,u_n) - N(t,u) \| \leq
     \| u_n - u \|+ \kappa_w \| F(u_n) - F(u) \|+
     \| \xi(t_n) - \xi(t) \| \leq \eps
   \]

  We now proceed to check the Lipschitz continuity of $N$. Using the Lipschitz
  continuity of $F$ (see \cref{lem:nemytskii}) we obtain
  for any $(t,u), (t,v) \in Q_b$ 
  \[
      \| N(t,u) - N(t,v) \| 
	\leq \| u - v \| + \| W \| \| F(u) - F(v) \| 
	\leq (1 + \kappa_w \| f' \|_\infty) \| u - v \|
  \]
  therefore $N$ is Lipschitz continuous in the second
  argument, uniformly with respect to the first, because its Lipschitz constant is
  independent of $t$.
\end{proof}

\section{Proof of \texorpdfstring{\cref{thm:existenceProjCauchy}}{Existence solution Projected Cauchy}}\label{app:proof:thm:existenceProjCauchy}
\begin{proof}
  The proof follows closely the steps in \cref{lem:continuityN,thm:existence}. Let
  \[
    Q_b = 
  \{ (t,u) \in \RSet \times \XSet_n \colon |t - t_0| \leq T, \quad \| u -
    u_0 \| \leq b \}.
  \]
  The boundedness of $P_n$ implies that $P_n N \colon Q_b \to \XSet_n$ is continuous,
  and Lipschitz continuous in its second argument, uniformly with respect to the
  first. Applying Theorem 5.2.4 in \cite{atkinson2005theoretical} we obtain the
  existence and uniqueness of a solution $u_n \in C^1(J,\XSet_n)$ to
  \cref{eq:projCauchy} with initial condition $u_n(t_0) = P_n u_0 \in \XSet_n$.
\end{proof}

\section{Proof of \texorpdfstring{\cref{lem:boundUnPrime}}{Lemma Bound on solution's
derivative}}\label{app:proof:lem:boundUnPrime}
\begin{proof}
    From the evolution equation \cref{eq:projCauchy} and \cref{lem:nemytskii} we
    obtain
    \[
      \| u'_n(t) \| \leq \| u_n(t) \| + \kappa_\Omega \| P_n W \| \| f \|_\infty +
      \| P_n \xi(t) \|,
    \]
    and \cref{eq:boundUnPrime} follows from $u_n, P_n \xi \in C(J,\XSet)$. The
    hypotheses on $W$ and $F$ guarantee that the operator 
    $K \colon \XSet \to \XSet$, $u \mapsto WF(u)$ is Fréchet differentiable with
    derivative 
    \[
      K'(u) \colon \XSet \to \XSet, \qquad v \mapsto \int_\Omega W(\blank,y) f'(u(y))
      v(y)\,dy.
    \]
    We obtain
    \[
      u''_n(t) = -u'_n(t) + P_n K'(u_n(t)) u'_n(t) + P_n \xi'(t),
    \]
    hence
    \[
      \| u''_n(t) \| \leq L_n \| u'_n (t) \| + \| P_n \xi'(t) \|
    \]
    and \cref{eq:boundUnDoublePrime} holds because $u'_n, P_n \xi' \in C(J,\XSet)$.
    The bound \cref{eq:boundUnDoublePrime} implies $u_n \in C^2(J,\XSet)$.

    The existence of $\kappa_1,\kappa_2$ in \cref{eq:boundUniformUnDiff} follows
    directly if $P_n v \to v$ for all $v \in \XSet$. Under this hypothesis: (i) by
    \cref{thm:convergence} $u_n \to u$ in ${C(J,\XSet)}$, hence the sequence with
    elements $\| u_n \|_{C(J,\XSet)}$ is bounded; (ii) the 
    same arguments used in the proof of \cref{thm:convergence}, in the discussion
    leading to \cref{eq:boundSpeed}, give $\| W - P_n W \| \to 0$, hence the
    sequences $\{\gamma_n\}$ and $\{L_n\}$ are bounded; (iii) $\xi \in C(J,\XSet)$
    implies the uniform convergence of $\{P_n \xi\}$ to $\xi$, hence the sequence
    with elements $\|P_n \xi \|_{C(J,\XSet)}$ is bounded; (iv) a similar argument on
    $\xi' \in C(J,\XSet)$ implies that the sequence with elements $\|P_n \xi
    \|_{C(J,\XSet)}$ is bounded.
\end{proof}

\section{Implementation of discrete schemes}\label{app:implementation}
\begin{table}\label{tab:exactProblems}
  \caption{Parameters for test neural field problems with exact solutions.}
  \label{tab:parameters}
  \centering
  \begin{tabular}{lccccc}
    \toprule
    Problem & $D$ & $\gamma$ & $k$ & $\theta$ & $\zeta(y)$ or $\zeta_p(y)$ \\
    \otoprule
    P1      &  0.8 &  0.5 &  5     & 0.3 & $ \exp{y} \cos y  $ \\
    P2      &  -   &  -   &  -     & -   & $ y^{20} $ \\
    P3      &  -   &  -   &  -     & -   & $ (1+16y^2)^{-1}$  \\
    P4      &  -   &  -   &  -     & -   & $ \exp (-y^2) $  \\
    P5      &  -   &  -   &  -     & -   & $ \exp (-y)$  \\
    P6      &  -   &  -   &  -     & -   & $ |y|^3 $  \\
    P7p     &  -   &  -   &  -     & -   & $ \cos^2 y $  \\
    P8p     &  -   &  -   &  -     & -   & $ (1+16 \cos^2 y)^{-1}$  \\
    P9p     &  -   &  -   &  -     & -   & $ |\cos^3 y|$  \\
    P10p    &  -   &  -   &  -     & -   & $ \cos^{20} y$ \\
    \bottomrule
  \end{tabular}
 \end{table}

\subsection{Finite Element Collocation with Trapezium quadrature} In a first
numerical test we implemented the method described in \cref{sec:FECollExample}. Since
the scheme converges as $O(h_x^2)$, we selected the composite trapezium quadrature
scheme, which preserves this order of accuracy. This leads to the set of ODEs
\begin{equation}\label{eq:discreteCollScheme}
  a'_i(t) = - a_i(t) + \sum_{j=0}^n W_{ij} f(a_j(t)) + \xi(x_i,t), \qquad a_i(0) = u_0(x_i,0),
  \qquad i \in \ZSet_{n+1}
\end{equation}
where
\[
  W_{ij} = w(x_i,x_j)\rho_j, \qquad \rho_j = 
  \begin{cases}
    h_x   & \text{for $j \in \NSet_{n-1}$,} \\
    h_x/2 & \text{for $j \in \{0,n\}$.} 
  \end{cases}
\] 

\subsection{Chebyshev Spectral Collocation with Trapezium and Clenshaw--Curtis quadrature} 
We have used Trapezium and Clenshaw--Curtis quadrature for the spectral collocation
scheme of \cref{sec:SpecCollExample}. In the former, we discretised the integrals in
\cref{eq:SpecCollCheb} using the composite trapezium rule, and arriving at a discrete
system analogous to \cref{eq:discreteCollScheme}, but where $\{x_i\}$ indicate the
$n+1$ Chebyshev nodes \cref{eq:ChebyInterpBasis}, and where a set of different $n+1$
quadrature nodes $\{z_i\}$ are taken to be evenly spaced by $h_x$, giving
\[
  W_{ij} = w(x_i,z_j)\rho_j, \qquad \rho_j = 
  \begin{cases}
    h_x   & \text{for $j \in \NSet_{n-1}$,} \\
    h_x/2 & \text{for $j \in \{0,n\}$.} 
  \end{cases}
\] 

Secondly, we implemented a scheme with Clenshaw--Curtis quadrature. In this case, we
form the matrix with elements $W(x_i,x_j) \rho_j$, where $\{ x_i \}$ are
the Chebyshev nodes and $\rho_j$ the Clenshaw--Curtis weights. The scheme is written as
\begin{equation}
  \bm{a}'(t) = - \bm{a}(t) + \bm{W} \bm{f}(\bm{a}(t)) + \bm{\xi}(t), \qquad \bm{a}(0)
  = \bm{u_0}(0).
\end{equation}
The matrix-vector product on the right-hand side, however, is evaluated calling $n+1$
times the FFT of an $(n+1)$-vector (see the accompanying
codes~\cite{avitabile2021-codes} where we have adapted for integrals the scripts in
\cite{Trefethen2008}).

\subsection{Finite Element Galerkin scheme with Gauss quadrature} 
\label{app:implementation:FEGalerkin}.
We used P1-P6 to test the Finite Element
Galerkin scheme with piecewise-linear hat functions, discussed in
\cref{sec:FEGalerkinExample}. The spatially-continuous scheme contains a sparse mass matrix
with entries $M_{ij} = \langle \ell_i,\ell_j \rangle_{L^2(-1,1)}$, computable in closed form~\cite[Section
3.3.1]{atkinson1997}. The scheme is 
\[
  \begin{aligned}
  \sum_{j=0}^n M_{ij} a'_j(t) = - \sum_{j=0}^n M_{ij} a_j(t) 
    & + \int_{-1}^1 \ell_i(x) \int_{-1}^1 w(x,y) f\Big(\sum_{k=0}^n a_k(t) \ell_k(y) \Big) \, dy \, dx\\
    & + \int_{-1}^1 \ell_i(x) \xi(x,t) \, dx \\
  \end{aligned}
\]
Since, by the projector error, this scheme converges to $O(h_x^2)$, one possibility to
obtain a matching discrete method is to use the composite Trapezium rule. In this
case, it is advantageous to pair it to a so-called \textit{mass-lumping procedure}
\cite[Section 13.3, page 595]{quarteroni2010numerical}, which uses the Trapezium rule also to evaluate the
components of $M_{ij}$. Since by the Trapezium rule $M_{ij} = \rho_i
\delta_{ij} + O(h_x^2)$, this allows us to pass from a sparse mass matrix, which must
be inverted to evaluate the right-hand side, to a new problem with an approximate,
but \textit{diagonal} mass matrix,
\[
\rho_i a'_i(t) = - \rho_i a_i(t) + \rho_i \sum_{j=0}^n W_{ij} f(a_j(t)) + \rho_i \xi(x_i,t), \qquad a_i(0) = u_0(x_i,0),
  \qquad i \in \ZSet_{n+1}
\]
where $\rho_i$ are the composite Trapezium weights. This discrete scheme uses the
fact that the inner products and integral operator on the right-hand side are
approximated at $O(h_x^2)$, that the projection scheme (before discretisation)
converges to $O(h_x^2)$, and hence one can tolerate the same error on $M_{ij}$. The main
advantage is that, once each equation is divided by the nonzero weights $\rho_i$,
this scheme is identical to the discrete Finite Element collocation
scheme~\cref{eq:discreteCollScheme}, therefore it does not require, in practice, any
inner product integration. Numerical convergence results for this scheme are
therefore given in \cref{fig:CollTests}(a).

An alternative is to proceed as in classical Finite Element methods, and pair the
hat functions $\{\ell_i\}$ with a Gaussian quadrature rule
\[
  \int_{-1}^1 \psi(x) \, dx \approx \sum_{q=1}^{n_q} \psi(z_q) \nu_q.
\]
For this scheme one introduces reference hat functions and coordinate mappings
\[
  \phi_{\pm}(z) = \frac{1\pm z}{2}, \qquad g_i: [-1,1] \to \Omega_i, \quad g_i(z) = x_i +
  \frac{(1+z)(x_{i+1}-x_i)}{2}, i \in \NSet_n
\]
and derive the following discrete scheme
\[
  \sum_{j=0}^n M_{ij} a'_j(t) = - \sum_{j=0}^n M_{ij} a_j(t) + 
  \frac{h}{2}\sum_{q=1}^{n_q} 
  \bigg[
    \phi_-(z_q) v(g_i(z_q),t) \nu_q + \phi_+(z_q) v(g_{i+1}(z_q),t) \nu_q
  \bigg]
\]
where 
\[
  v(x,t) = \xi(x,t) + \frac{h}{2} \sum_{j=1}^{n} \sum_{q=1}^{n_q} w(x,g_j(z_q)) f( a_j
  \phi^-(z_q) + a_{j+1} \phi^+(z_q)) \nu_q.
\]
This scheme has been implemented for a Gaussian quadrature scheme with $n_q=2$, and
its convergence properties are seen in \cref{fig:GalerkinTests}(a).

\subsection{Spectral Galerkin Scheme with Trapezium Quadrature} As a final test,
we implemented the spectral Galerkin scheme of \cref{sec:SpectralGelerkinExample}. We
rewrite the scheme as
\begin{equation} \label{eq:SpecGalSchemeV}
  \begin{aligned}
  a'_{i}(t) 
   = 
   &- a_{i}(t)
   + \int_{\Omega} \phi_i(x) v^*(x,t)\, dx \\
 a_i(t_0) = & \int_{\Omega} \phi_i(x) u^*_0(x) \, dx
  \end{aligned}
\end{equation}
where
\[
  v (x,t) = \xi(x,t) +
   \int_\Omega w(x,y) f\bigg( \sum_{j \in \ZSet_{n+1}} a_j(t)
 \phi_j(y)\bigg) \, dy \,dx 
\]
Since the integrand in the expression for $v$ is periodic, we selected for this
scheme a composite trapezium rule, which is well suited for periodic
integrands~\cite{trefethen2014exponentially}. In addition, the integrals in
\cref{eq:SpecGalSchemeV} cam be evaluated using the FFT. We have
implemented this scheme combining Fast Fourier Transform with a pesudospectral
evaluation of the integrands in $v$. More specifically, the scheme can be expressed
compactly in vector notation, using forward and backward Discrete Fourier Transforms
operators for vectors, as follows
\[
  \begin{aligned}
    \bm{a}'(t) & = - \bm{a}(t) + \mathcal{F}\Big[ \bm{\xi}(t) 
               + h_x \bm{W} \bm{f}\big( \mathcal{F}^{-1}[\bm{a}(t)] \big) \Big], \\
      \bm{a}(0) & = \mathcal{F}[ \bm{u}_0 ].
  \end{aligned}
\]

\bibliographystyle{siamplain}
\bibliography{references}
\end{document}